\documentclass[reqno]{amsart}
\usepackage{geometry}                		
\usepackage{graphicx}				
\usepackage{amssymb, color}
\usepackage{amsmath}
\usepackage[all]{xy}
\usepackage{array} 

\theoremstyle{plain}
\usepackage{latexsym}
\usepackage{amsthm}
\usepackage{amsfonts}
\usepackage{mathrsfs}
\usepackage{lscape}
\usepackage{enumerate}
\usepackage{pst-node}
\usepackage{tikz-cd}
\usepackage{stmaryrd}
\usepackage{epsfig}
\usepackage[sc]{mathpazo}
\usepackage{verbatim}
\usepackage[latin1]{inputenc}              
\usepackage[T1]{fontenc}
\usepackage{tikz}
\usetikzlibrary{positioning}
\newtheorem{theorem}{Theorem}
\numberwithin{theorem}{section}
\newtheorem{definition}[theorem]{Definition}

\newtheorem{proposition}[theorem]{Proposition}
\newtheorem{lemma}[theorem]{Lemma}
\newtheorem{corollary}[theorem]{Corollary}

\numberwithin{equation}{section}

\newcommand{\cblue}[1]{{\textcolor{black}{#1}}} 
\newcommand{\ccblue}[1]{{\textcolor{black}{#1}}}
\newcommand{\ccred}[1]{{\textcolor{black}{#1}}}
\newcommand{\cgreen}[1]{{\textcolor{black}{#1}}}
\newcommand{\ccgreen}[1]{{\textcolor{black}{#1}}}

\title{D-critical locus structure on the Hilbert schemes of some local toric Calabi-Yau threefolds}
\author{Sheldon Katz}
\address{Department of Mathematics\\University of Illinois at Urbana-Champaign\\Urbana IL 61801 USA}
\email{katzs@illinois.edu}
\author{Yun Shi}
\address{Department of Mathematics\\Brandeis University\\Waltham MA 02453 USA}
\email{yunshi@brandeis.edu}

\begin{document}
\maketitle{}

\begin{abstract}
The notion of a d-critical locus is an ingredient in the definition of motivic Donaldson-Thomas invariants by \cite{BJM19}. In this paper we show that there is a d-critical locus structure on the Hilbert scheme of dimension zero subschemes on some local toric Calabi-Yau 3-folds. We also show that using this d-critical locus structure and a choice of orientation data, the resulting motivic invariants agree with the definition given by the previous work of \cite{BBS13}.


{\bf Subject:} 14N35
\end{abstract}

\section{Introduction}

Donaldson-Thomas (DT) theory was introduced in \cite{Tho00} as an enumerative theory which gives a virtual count of stable coherent sheaves with fixed topological invariants on certain 3-folds, including Calabi-Yau threefolds.  In the Calabi-Yau case, this moduli problem supports a perfect obstruction theory of virtual dimension zero in good situations, and so defines a degree zero class in the Chow ring of the moduli space --- the virtual fundamental class of \cite{BF97}.  \ccblue{When the moduli space is proper,} the DT invariant is defined as the degree of this virtual fundamental class.  This theory was applied to Hilbert schemes of points in \cite{Tho00} by identifying the Hilbert scheme with the moduli space of their corresponding ideal sheaves. 

It turns out that there is a rich structure underlying the DT invariant. The DT moduli space supports a symmetric obstruction theory, which implies that \ccblue{when the moduli space is proper} the DT invariant is a weighted Euler characteristic of the moduli space, the weighting being given by a constructible function on the moduli space called the Behrend function \cite{Beh09}.  These ideas are applied to show that for a projective threefold $X$, the generating function of the DT invariants of $\operatorname{Hilb}^n(X)$ can be expressed in terms of the MacMahon function and the Euler characteristic of $X$ \cite{BF08, LP09, MNOP06}.
 
If the moduli space can be globally realized as a degeneracy locus, the value of its Behrend function at a point can be expressed in terms of the Euler characteristic of the Milnor fiber at that point \cite{Beh09,PP98}.  Replacing the Milnor fibers with their motivic incarnations, one is led to the notion of Motivic DT invariants. 

Motivic DT theory was introduced and developed in \cite{KS,BJM19}, and had been studied by several authors, see e.g.\ \cite{BBS13,Dav,MMNS12,Nag}. One of the first computations of motivic DT invariants was given for $\operatorname{Hilb}^n(X)$, where $X$ is a Calabi-Yau threefold  \cite{BBS13}. This computation relied on the fact that $\operatorname{Hilb}^n(\mathbb{C}^3)$ can be realized as a degeneracy locus with an equivariant torus action. The authors introduced a notion of motivic DT invariants of $\operatorname{Hilb}^n(X)$ based on their computations for $\operatorname{Hilb}^n(\mathbb{C}^3)$.

Later, a general formalism for motivic DT invariants was given in \cite{BJM19}, using the notions of a d-critical locus and orientation  introduced in \cite{Joy15}. The notion of orientation is extracted from the notion of orientation data introduced in \cite{KS}.   

Now let $X$ be a local toric surface $\omega_S$, the total space of the canonical bundle of a smooth, complete toric surface $S$.
In this paper, we define a d-critical locus structure on $\operatorname{Hilb}^n(X)$ if $X=\omega_{\mathbb{P}^2}$ or $X=\omega_{\mathbb{F}_n}$, and endow it with an orientation following \ccblue{\cite{Dav}}.  We then show that the motivic DT invariants of $\operatorname{Hilb}^n(X)$ as defined using the d-critical locus and orientation agree with the motivic DT invariants as defined in \cite{BBS13}. We therefore show that the results of \cite{BBS13} are valid in the context of the more general theory for these Calabi-Yau threefolds.

A natural question arises which we do not address further in this paper:


\smallskip\noindent
The notion of a d-critical locus was introduced as a classical truncation of the notion of a $-1$-shifted symplectic structure in derived algebraic geometry.  It was shown in \cite{PTVV13} that the derived moduli stack of sheaves on a \ccblue{projective} Calabi-Yau threefold has a canonical $-1$-shifted symplectic structure. \ccblue{For a similar result in the noncompact case, see \cite{BD19}.} Does the classical truncation of the canonical $-1$-shifted symplectic structure on the derived Hilbert scheme agree with our d-critical locus structure? This would imply that our motivic result \ccgreen{Corollary~\ref{cor:toricBBS} can be extended to} the total space of the canonical bundle of any smooth toric surface.\footnote{\cblue{While this paper was under review, we have  answered this question in the affirmative \cite{KS22}}.}

\subsection{Outline of the paper}
In Section 2 we review background material including the definition of motivic DT invariants following \cite{BJM19}.  Section 3 begins with the \ccred{necessary} toric geometry.  We then prove that there is a canonical  d-critical locus structure on an open subset of $\operatorname{Hilb}^n(\omega_S)$ for any smooth toric surface $S$ (Proposition~\ref{prop:opendcrit}).  We use this result and its proof to show that there is a natural d-critical locus structure on all of $\operatorname{Hilb}^n(\omega_S)$ induced by copies of $\operatorname{Hilb}^n(\mathbb{C}^3)$'s, for $S=\mathbb{P}^2$ and $S=\mathbb{F}_n$ (Theorem~\ref{thm:dcrit}). In Section 4, we exhibit an orientation on $\operatorname{Hilb}^n(\omega_S)$ associated to this d-critical structure. Finally in Section 5, using the d-critical locus structure constructed in Section 3 and the orientation in Section 4, we recover the computation of $\operatorname{Hilb}^n(\omega_S)$ in \cite{BBS13}.

\subsection{Notations}
All schemes in this paper are assumed to be separated and of finite type over $\mathbb{C}$. We use $X$ to denote a smooth quasi-projective Calabi-Yau 3-fold. For a smooth surface $S$, we denote the total space of its canonical bundle by $\omega_S$, and the projection from $\omega_S$ to $S$ by $\pi: \omega_S\rightarrow S$.

\section{Background}

\subsection{Introduction to motivic DT invariants}

In this section, we review the main ingredients of motivic Donaldson-Thomas (DT) invariants as defined by Bussi, Joyce and Meinhardt \cite{BJM19}. 

First we recall the definition of the monodromic Grothendieck ring, where the motivic invariant associated to a moduli space lives. Let $S$ be a separated scheme of finite type over $\mathbb{C}$. The Grothendieck group of $S$-schemes is the abelian group generated by isomorphism classes of $S$-schemes $[X\rightarrow S]$, modulo the scissor relations:
\begin{equation*}
[T\to S]=[T'\to S]+[T\textbackslash T'\to S]
\end{equation*}
for $T'$ a closed $S$-subscheme of $T$.

Now we consider an equivariant version of the Grothendieck ring. Let $\mu_n$ be the group of $n$th roots of unity. A $\mu_n$ action on a scheme $T$ is called \emph{good} if every $\mu_n$ orbit is contained in an affine subscheme of $T$. Note that there are maps
$\mu_m\rightarrow \mu_n$ if $n|m$,
defined by $x\mapsto x^{m/n}$ forming an inverse system. 
Denote $\varprojlim \mu_n$ by $\hat{\mu}$. 
\begin{definition} (\ccblue{\cite{DF00, BJM19}}) \label{def:grothendieck}
The monodromic Grothendieck group $K_0^{\hat{\mu}}(S)$ is the abelian group generated by isomorphism classes $[T\to S, \hat{\sigma}]$, where $T$ is a separated $S$-scheme of finite type and 
\begin{equation*}
\hat{\sigma}:\hat{\mu}\times T\rightarrow T
\end{equation*} is a $\hat{\mu}$ action which factors through a good $\mu_n$ action \ccblue{for some $n$}. The relations of $K_0^{\hat{\mu}}(S)$ are given by

(i) $[T_1, \hat{\sigma_1}]= [T_2, \hat{\sigma_2}]$ for $T_1$ and $T_2$ equivariantly isomorphic as $S$-schemes,

(ii) $[T, \hat{\sigma}]= [T', \hat{\sigma}|_{T'}]+[T\setminus T', \hat{\sigma}|_{T\setminus T',}]$ for any $T'\subset T$ closed $\hat{\sigma}$ invariant $S$-subscheme,

(iii) $[T\times \mathbb{A}^n, \hat{\sigma}\times\hat{\tau}_1]= [T\times \mathbb{A}^n, \hat{\sigma}\times\hat{\tau}_2]$ for any linear $\hat{\mu}$ actions $\hat{\tau}_1$, $\hat{\tau}_2$ on $\mathbb{A}^n$.  
\end{definition}

Note that there is an obvious commutative ring structure on $K_0^{\hat{\mu}}(S)$ defined by 
\begin{equation*}
[T_1, \hat{\sigma}]\cdot[T_2, \hat{\tau}]=[T_1\times_S T_2, \hat{\sigma}\times \hat{\tau}].
\end{equation*} 
We denote by $\mathbb{L}$ the element $[\mathbb{A}^1\times S\rightarrow S,\hat{\iota}]$, where $\hat{\iota}$ is the trivial action. Then we obtain a ring
\begin{equation*}
\mathcal{M}^{\hat{\mu}}_S:=K_0^{\hat{\mu}}(S)[\mathbb{L}^{-1}]
\end{equation*}
by formally joining the inverse of $\mathbb{L}$ under this multiplication. 
\ccgreen{We can similarly define $\mathcal{M}_S$  as the abelian group generated by classes $[T\to S]$ of $S$-schemes  of finite type, modulo the relations analogous to (i) and (ii) of Definition~\ref{def:grothendieck}.}
\ccblue{ Note that there is a map $\pi_{\hat{\mu}}: \mathcal{M}^{\hat{\mu}}_S\rightarrow \mathcal{M}_S$ of $\mathcal{M}_S$\ccgreen{- modules}, defined by taking generators to the corresponding orbit space: 
\begin{equation*}
\pi_{\hat{\mu}}: [T\rightarrow S, \hat{\sigma}]\rightarrow [T/{\hat{\sigma}}].
\end{equation*}
Let $f:S_1\rightarrow S_2$ be a morphism between $\mathbb{C}$-schemes. The notions of pushforward and pullback on $K_0(S)$ are defined by $f_*([T\rightarrow S_1])=[T\rightarrow S_2]$ by composing the map $T\rightarrow S_1$ with the map $f$, and $f^*[T'\rightarrow S_2]=[T'\times_{S_2}S_1\rightarrow S_1]$. \ccgreen{These maps} induce a pushforward and a pullback on $\mathcal{M}^{\hat{\mu}}_S$ in the obvious way.}

There is a less obvious product on $K^{\hat{\mu}}_0(S)$ or $\mathcal{M}^{\hat{\mu}}_S$, called the convolution product. It is denoted by '$\odot$' in \cite{BJM19}. Since we will not explicitly use it in this paper, we refer to \cite{Loo02}, \cite{DF99} and \cite{BJM19} for its definition.

The motivic DT invariants we consider here take values in a quotient of $(\mathcal{M}^{\hat{\mu}}_S, \odot)$ denoted by $(\overline{\mathcal{M}}^{\hat{\mu}}_S, \odot)$. To define it, we first write
\begin{equation*}
\mathbb{L}^{1/2}=[S,\hat{\iota}]-[S\times\mu_2,\mu_2]
\end{equation*}
with the natural $\mu_2$ action.  The notation is justified since the square of the right hand side (using $\odot$) can be checked to be $\mathbb{L}$.  It follows immediately that $\mathbb{L}^{-1/2}\in \mathcal{M}^{\hat{\mu}}_S$ as well.  Then to a
principal $\mathbb{Z}_2$\ccgreen{-bundle} $P$ over a scheme $T$ we associate a motive:
\begin{equation*}
\Upsilon(P)=\mathbb{L}^{-1/2}\odot([T, \hat{\iota}]-[P, \hat{\rho}])\in \mathcal{M}^{\hat{\mu}}_T \hspace{1mm},
\end{equation*} 
where $\hat{\rho}$ is induced by the $\mathbb{Z}_2$ action on $P$. Then $(\overline{\mathcal{M}}^{\hat{\mu}}_S, \odot)$ is obtained from $(\mathcal{M}^{\hat{\mu}}_S, \odot)$ by taking the quotient by the ideal generated by elements $\phi_*(\Upsilon(P\otimes_{\mathbb{Z}_2}Q)-\Upsilon(P)\odot\Upsilon(Q))$, for all $\mathbb{C}$\ccgreen{-scheme} morphisms $\phi: Z\rightarrow S$ and $P$, $Q$ principal $\mathbb{Z}_2$\ccgreen{-bundles} on $Z$.  Note that if $P$ is a trivial $\mathbb{Z}_2$\ccgreen{-bundle}, then $\Upsilon(P)=1$ is trivial as well.

\bigskip 

Now we review the definition of motivic DT invariants. First consider the case of a moduli space $Y$ which can be globally realized as a degeneracy locus, i.e. $Y:=\{df=0\}\subset U$ where $f:U\to \mathbb{C}$ is a regular function on a smooth scheme $U$. Then the motivic DT invariant can be defined as the motivic vanishing cycle \ccblue{up to a factor of  a power of $\mathbb{L}^{1/2}$}. We will not work with motivic vanishing cycles directly in this paper, and we simply refer to \cite{Loo02}, \cite{DF99}, \cite{BJM19} for their precise definition. 

More generally, the motivic invariants can still be defined if the moduli space can be covered by degeneracy loci which satisfy a compatibility condition. The precise \ccblue{terminology} for this structure is called a \emph{d-critical locus} structure \cite{Joy15}. 
Let $Y$ be a $\mathbb{C}$\ccgreen{-scheme} locally of finite type. The following theorem is from \cite{Joy15}:

\begin{theorem} (\cite{Joy15}, Theorem 2.1)
	\label{thm1}
There exists a sheaf $\mathcal{S}_Y$ of $\mathbb{C}$ vector spaces, uniquely characterized by two properties. 

(i) Suppose $R\subset Y$ is a Zariski open subset of $Y$, and $i:R\hookrightarrow U$ a closed embedding in some smooth scheme $U$. Define the sheaf of ideals $I_{R, U}$ by the following exact sequence of vector spaces on $R$.
\begin{equation*}
0\rightarrow I_{R, U}\rightarrow i^{-1}(O_U)\rightarrow O_Y|_R\rightarrow 0.
\end{equation*}

Then there is an exact sequence of sheaves of vector spaces on $R$:
\begin{equation*}
0\rightarrow \mathcal{S}_Y|_R\xrightarrow{\iota_{R, U}} \frac{i^{-1}(O_U)}{I_{R, U}^2}\xrightarrow{d}\frac{i^{-1}(T^*U)}{I_{R, U}\cdot i^{-1}(T^*U)}\hspace{1mm},
\end{equation*}
where $\iota_{R, U}$ is a morphism of sheaves of vector spaces, and $d$ is the differential map. 

(ii) Let $R\subset S\subset Y$ be Zariski open inclusions, and $i: R\hookrightarrow U$, $j: S\hookrightarrow V$ closed embeddings in smooth schemes $U$ and $V$. Let $\Phi: U\to V$ be a morphism satisfying $\Phi\circ i=j|_R$. Then the following diagram commutes:
\[\begin{tikzcd}
0\arrow{r} & \mathcal{S}_Y|_R \arrow{r}{\iota_{S, V}|_R} \arrow[swap]{d}{id} & \frac{j^{-1}(O_V)}{I^2_{S, V}}|_R\arrow{r}{d} \arrow{d}{i^{-1}(\Phi^\sharp)} & \frac{j^{-1}(T^*V)}{I_{S, V}\cdot j^{-1}(T^*V)}|_R \arrow{d}{i^{-1}(d\Phi)}\\
0\arrow{r} & \mathcal{S}_Y|_R \arrow{r}{\iota_{R, U}} & \frac{i^{-1}(O_U)}{I^2_{R, U}}\arrow{r}{d} & \frac{i^{-1}(T^*U)}{I_{R, U}\cdot i^{-1}(T^*U)}\hspace{1mm}.
\end{tikzcd}
\]
\end{theorem}
Let $\mathcal{S}_Y^0$ be the kernel of the composition 
\begin{equation*}
\mathcal{S}_Y\rightarrow O_Y\rightarrow O_{Y_{red}}\hspace{1mm},
\end{equation*}
\ccblue{where the map $\mathcal{S}_Y\rightarrow O_Y$ is locally defined by composing $\iota_{R, U}$ with $i^{-1}(O_U)\rightarrow O_Y|_R$. }

Then the sheaf $\mathcal{S}_Y$ has a canonical decomposition 
\begin{equation*}
\mathcal{S}_Y\simeq \mathbb{C}_Y\oplus \mathcal{S}_Y^0\hspace{1mm}.
\end{equation*}

\begin{definition} (\cite{Joy15} Definition 2.5)
\label{dcrit}
An algebraic d-critical locus over $\mathbb{C}$ is a pair $(Y, s)$, where $Y$ is a $\mathbb{C}$\ccgreen{-scheme} and $s\in H^0(\mathcal{S}_Y^0)$ such that the following is satisfied: for every point $y\in Y$, there is a Zariski open neighborhood $R$ of $y$ with a closed embedding $i:R\hookrightarrow U$ into a smooth scheme $U$, such that $i(R)=\{df=0\}\subset U$ for $f: U\rightarrow \mathbb{C}$ a regular function on $U$. Furthermore, $\iota_{R, U}(s|_R)=i^{-1}(f)+I_{R, U}^2$. 
\end{definition}
Using the notation in the definition, the charts $(R, U, f, i)$'s are called critical charts of $(Y, s)$.

Consider a moduli space which has a d-critical locus structure. In particular it is locally covered by degeneracy loci, and one can associate motivic vanishing cycles to the local charts. However, these locally defined invariants do not naturally agree on intersections, see e.g. \cite{BJM19} Example 3.5.  
\ccblue{It turns out} that these locally defined invariants can be modified using an orientation, and they then agree and can be glued to a global motivic invariant.

To give the definition of an orientation, we first recall the notion of the virtual canonical bundle of a d-critical locus.
\begin{theorem} (\cite{Joy15} Theorem 2.28, \cblue{\cite{BJM19} Theorem 5.4})
\label{defK}
Let $(Y, s)$ be a d-critical locus, and let $Y^{red}$ be $Y$ with its reduced structure. Then there exists a line bundle $K_{Y, s}$ on $Y^{red}$ which is uniquely defined by two properties. 

(i) If $(R, U, f, i)$ is a critical chart on $(Y, s)$, there is a natural isomorphism
\begin{equation}
\label{equiota}
\iota_{R, U, f, i}: K_{Y, s}|_{R^{red}}\rightarrow i^*(K_U^{\otimes 2})|_{R^{red}},
\end{equation}
where $K_U$ is the usual canonical bundle of $U$.
 
(ii) Let $\Phi: (R, U, f, i)\rightarrow (S, V, g, j)$ be an embedding of critical charts on $(Y, s)$.  Then there is an isomorphism of line bundles on $Crit(f)^{red}$:
\begin{equation*}
J_\Phi:K^{\otimes^2}_U|_{Crit(f)^{red}}\rightarrow \Phi|^*_{Crit(f)^{red}}(K_V^{\otimes^2}).
\end{equation*}
such that 
\begin{equation*}
i|^*_{R^{red}}(J_\Phi):i^*(K_U^{\otimes^2})|_{R^{red}}\rightarrow j^*(K_V^{\otimes^2})|_{R^{red}},
\end{equation*}
is an isomorphism and 
\begin{equation*}
\iota_{S, V, g, j}|_{R^{red}}=i|^*_{R^{red}}(J_\Phi)\circ \iota_{R, U, f, i}: K_{X, s}|_{R^{red}}\rightarrow j^*(K_V^{\otimes^2})|_{R^{red}}.
\end{equation*}
\end{theorem}
Here $J_\Phi$ is induced by the isomorphism $K_U\otimes \Lambda^nN^*_{UV}\rightarrow \Phi^*(K_V)$.  For details of the definition of $J_\Phi$ see \cite{Joy15}, Def. 2.26 or \cite{BJM19}, section 4.

\begin{definition} (\cite{Joy15} Definition 2.31)
\label{defOD}
Let $(Y, s)$ be a d-critical locus. An \emph{orientation on $(Y,s)$} is given by a line bundle $L$ on $Y^{red}$ \ccred{together with an isomorphism $L^{\otimes2}\simeq K_{Y, s}$}.
\end{definition}

With these two additional structures, a global motivic DT invariant can always be defined:

\begin{theorem} (\cite{BJM19} Theorem 5.10)
\label{thm1.1}
Let $(Y,s)$ be a finite type algebraic d-critical locus with a choice of orientation $K^{1/2}_{Y,s}$. There exists a unique motive $MF_{Y, s}\in \overline{\mathcal{M}}_Y^{\hat{\mu}}$ with the property that if $(R, U, f, i)$ is a critical chart on $(Y, s)$, then
\begin{equation*}
MF_{Y, s}|_R=i^*(\mathbb{L}^{-dimU/2}\odot MF^\phi_{U,f})\odot\Upsilon(Q_{R, U, f, i}) \in \overline{\mathcal{M}}_R^{\hat{\mu}}.
\end{equation*}
\end{theorem}

Here, the term $MF^\phi_{U,f}$ is the motivic vanishing cycle associated to the critical chart $(R, U, f, i)$, \ccred{see \cite{Loo02}, \cite{DF99}, \cite{BJM19}}. The term $Q_{R, U, f, i}$ is the principal $\mathbb{Z}_2$\ccgreen{-bundle} of local isomorphisms $\alpha_R:K_{Y, s}^{1/2}|_{R^{red}}\to i^*(K_U)|_{R^{red}}$ such that $\alpha_R^{\otimes 2}=\iota_{R, U, f, i}$ in (\ref{equiota}). 
We would like to point out that the motivic vanishing cycle used in this paper is \cgreen{the class of $f^{-1}(0)$ with the trivial action minus the motivic nearby cycle, which is the negative of the definition used in \cite{BBS13}}.
\subsection{Construction of $\operatorname{Hilb}^n(\mathbb{C}^3)$}\label{subsec:}

We next briefly recall the construction of $\operatorname{Hilb}^n(\mathbb{C}^3)$ used in \cite{BBS13}, in particular its presentation as a degeneracy locus.

To a subscheme $Q\subset \mathbb{C}^3$ with Hilbert polynomial $P_Q=n$ we can associate 
an $n$-dimensional vector space $V_n=H^0(\ccgreen{O}_Q)$, three pairwise commuting linear maps 
\begin{equation*}
X, Y, Z: V_n\to V_n
\end{equation*}
defined as multiplication by $x$, $y$, $z\in \mathbb{C}[x, y, z]$, and a vector $v\in V_n$ corresponding to $1\in H^0(\ccgreen{O}_Q)$.  The vector is $v$ cyclic for the action of $\mathbb{C}[X, Y, Z]$ on $V_n$: $\mathbb{C}[X,Y,Z]\cdot v=V_n$.  

Now consider the space of triples of $n\times n$ matrices and a vector in $V_n$: 
\begin{equation*}
Hom(V_n, V_n)^3\times V_n. 
\end{equation*}
This space is a quasiprojective variety and admits a $GL_n$ action induced from the action of $GL_n$ on $V_n$. The character $\chi: GL_n\rightarrow \mathbb{C}^*$ defined by $\chi(g)=\operatorname{det}(g)$ defines a linearization of the trivial bundle. Let $U$ be the stable locus of the linearization. It turns out that $U$ consists of the points $(X, Y, Z, v)$ where $v$ is cyclic for the action of $X, Y, Z$. Consider the GIT quotient 
\begin{equation*}
\operatorname{NHilb}^n(\mathbb{C}^3):=Hom(V_n, V_n)^3\times V_n\sslash GL_n=U/GL_n
\end{equation*}
with respect to this linearization. 
Let $W:\operatorname{NHilb}^n(\mathbb{C}^3)\rightarrow \mathbb{C}$ be the function on $\operatorname{NHilb}^n(\mathbb{C}^3)$ defined by $W(X, Y, Z, v)=\operatorname{tr}([X, Y]Z)$. The condition $\{dW=0\}$ is equivalent to the condition that $X$, $Y$ and $Z$ pairwise commute. Then 
the locus $\{dW=0\}$ is isomorphic to $\operatorname{Hilb}^n(\mathbb{C}^3)$. As the construction suggests, we call $\operatorname{NHilb}^n(\mathbb{C}^3)$ the noncommutative Hilbert scheme. 

\medskip\noindent
{\bf Remark.} \label{lineartrans} While any nonzero scalar multiple of $W$ would work just as well as $W$ in defining $\operatorname{Hilb}^n(\mathbb{C}^3)$, the construction of \cite{BBS13} makes clear that the choice of Calabi-Yau form $dx\wedge dy\wedge dz$ on $\mathbb{C}^3$ determines the normalization of $W$ described above.  More precisely, any determinant 1 linear change of coordinates in $T=\mathrm{span}\{x,y,z\}$ will leave $W$ unchanged.  

\section{D-critical Locus Structure on $\operatorname{Hilb}^n(\omega_S)$}\label{sec:dcrit}

We begin this section with a lightning review of some standard notions and notations of toric geometry.  We knowingly omit some important details and simply refer the reader to \cite{CLS11} if more detail is needed.

We let $N$ be an $n$-dimensional lattice and $\Sigma$ a fan in $N_\mathbb{R}=N\otimes\mathbb{R}$, giving rise to a normal toric variety $X_\Sigma$ with torus $T=N\otimes_{\mathbb{Z}}\mathbb{C}^*$. We can describe $X_\Sigma$ by a quotient construction.  Let $v_1,\ldots, v_r$ be the primitive integral generators of the 1-dimensional cones in $\Sigma(1)$. \ccgreen{Let $A_{n-1}$ denote the Chow group of rational equivalence classes of cycles of dimension $n-1$.} We have an exceptional closed subset $Z(\Sigma)\subset\mathbb{C}^r$ and the group $G=Hom(A_{n-1}(X_\Sigma),\mathbb{C}^*)$ acts on $\mathbb{C}^r-Z(\Sigma)$ by an action which will be described below.  Then $X_\Sigma$ is identified with the quotient
\begin{equation*}
X_\Sigma=\left(\mathbb{C}^r-Z(\Sigma)
\right)/G.
\end{equation*}
If $\sigma$ is a cone in $\Sigma$, then we have an affine toric variety $X_\sigma$, which is identified with an open subset of $X_\Sigma$ by
\begin{equation*}
X_\sigma=\left\{[(x_1,\ldots,x_r)]\in X_\Sigma\mid x_i\ne 0{\rm \ if\ }v_i\not\in\sigma
\right\},
\end{equation*}
where $[(x_1,\ldots,x_r)]$ denotes equivalence class of $(x_1,\ldots,x_r)\in \mathbb{C}^r-Z(\Sigma)$ mod $G$.

For $i=1\ldots r$ we have $T$-invariant divisors $D_i\subset X_\Sigma$ defined by $x_i=0$, where $(x_1,\ldots,x_r)$ are coordinates on $\mathbb{C}^r$.  The action of $G$ on $\mathbb{C}^r$ is given by
\begin{equation*}
g\cdot(x_1,\ldots,x_r)=(g([D_1])x_1,\ldots,g([D_r])x_r)
\end{equation*}
where $[D_i]$ is the divisor class of $D_i$.

Now suppose that $X_\Sigma$ is smooth.  Then $A_{n-1}(X_\Sigma)$ is torsion-free and consequently $G\simeq(\mathbb{C}^*)^{r-n}$.  Furthermore, if $\sigma\in\Sigma(n)$ is a maximal cone generated by $v_{i_1},\ldots,v_{i_n}$, the affine toric variety $X_\sigma$ is isomorphic to $\mathbb{C}^n$ via the isomorphism
\begin{equation}\label{eq:affinecoordinates}
\mathbb{C}^n\to X_\sigma,\qquad (x_1,\ldots,x_n)\mapsto (1,\ldots,1,x_1,1,\ldots,1,x_2,1\ldots,1,x_n,1\ldots,1),
\end{equation}
where $x_j$ is in the $i_j$th position and all other entries are 1.

We are concerned with the case where $X_\Sigma$ is a smooth Calabi-Yau threefold.  For any smooth toric variety we have
\begin{equation*}
\ccgreen{O}(K_{X_\Sigma})\simeq\ccgreen{O}\left(-\sum_{i=1}^rD_i\right),
\end{equation*}
so the Calabi-Yau condition says that $\sum_{i=1}^rD_i\sim 0$.  It is convenient to express this condition in terms of the homogeneous coordinate ring
\begin{equation*}
S=\mathbb{C}[x_1,\ldots,x_r].
\end{equation*}
The ring $S$ is graded by $A_{n-1}(X_\Sigma)=A_2(X_\Sigma)$, with $x_i$ having degree $[D_i]\in A_2(X_\Sigma)$.  The Calabi-Yau condition translates into the statement that $x_1\cdots x_r$ has degree $0$.

We now construct a nowhere-vanishing $T$-invariant holomorphic form $\Omega$ on $X_\Sigma$, well-defined up to sign. We have holomorphic $r$-form on $\mathbb{C}^r$
\begin{equation*}
\widehat{\Omega}=dx_1\wedge\cdots\wedge dx_r,
\end{equation*}
well-defined up to a sign depending on the chosen ordering of $\Sigma(1)$.  The action of $G\simeq(\mathbb{C}^*)^{r-3}$ on $\mathbb{C}^r$ determines a rank $r-3$ lattice of holomorphic vector fields on $\mathbb{C}^r$.   More precisely, let $\Lambda\subset\mathbb{Z}^r$ be the lattice of relations
\begin{equation}\label{eqn:Lambda}
\Lambda=\left\{(a_1,\ldots,a_r)\in\mathbb{Z}^r\mid \sum a_i v_i=0
\right\}.
\end{equation}
Then $\lambda=(a_1,\ldots,a_r)$ determines the vector field
\begin{equation*}
\chi_\lambda=\sum a_i x_i\frac{\partial}{\partial x_i}
\end{equation*}
on $\mathbb{C}^r$.  The relation $\lambda$ determines a 1-parameter subgroup of $G$
\begin{equation*}
t\mapsto ([D_i]\mapsto t^{a_i}).
\end{equation*}
The fact that $\lambda=(a_1,\ldots,a_r)$ satisfies (\ref{eqn:Lambda}) implies that $[D_i]\mapsto t^{a_i}$ is compatible with linear equivalence and so describes an element of $G=Hom(A_{n-1}(X_\Sigma),\mathbb{C}^*)$.
The vector field $\chi_\lambda$ is then seen to be the infinitesimal action of this 1-parameter subgroup. 

Choosing generators
$\chi_1,\ldots,\chi_{r-3}$ for the lattice of vector fields, we define
\begin{equation}\label{eqn:omega}
\Omega = i_{\chi_1}\cdots i_{\chi_{r-3}}(\widehat{\Omega}),
\end{equation}
\ccred{where $i_{\chi_i}$ is contraction with respect to $\chi_i$.} For each $j=1,\ldots r-3$ we have $i_{\chi_j}\Omega=0$, so that $\Omega$ is $G$-invariant and therefore pulls back from a holomorphic 3-form on $X_\Sigma$.  We abuse notation and denote this 3-form on $X_\Sigma$ by $\Omega$ as well.  In other words, $\Omega$ is the unique holomorphic 3-form $\Omega$ on $X_\Sigma$ whose pullback to $\mathbb{C}^r-Z(\Sigma)$ is the right hand side of (\ref{eqn:omega}).  The Calabi-Yau condition guarantees that $\Omega$ is a 3-form rather than merely a 3-form valued in a line bundle. 

We now specialize further to local toric surfaces, the total space $\omega_S$ of the canonical bundle of a smooth complete toric surface $S$.  Since $S$ is a toric variety, we have a rank 2 lattice $N'$ and a complete fan $\Sigma'$ in $N'_\mathbb{R}$ with edges spanned by primitive integral vectors $v_1',\ldots,v_{r-1}'$.  We form the rank 3 lattice $N=N'\oplus \mathbb{Z}$, identify the fan $\Sigma'$ with the collection of cones $\{\sigma'\times\{1\}\mid \sigma'\in\Sigma'\}$ in the affine hyperplane of $N_\mathbb{R}$ defined by setting the last coordinate to 1.  We then form the fan $\Sigma$ as the collection of cones in $N_\mathbb{R}$ over the cones ${\sigma'}\times\{1\}$, \ccgreen{i.e.\ the collection of three-dimensional cones generated by the rays from the origin through the vertices of $\sigma'\times \{1\}$}.  In particular, the primitive integral generators of the 1-dimensional cones of $\Sigma$ are identified with
\begin{equation}\label{eq:edges}
v_1=(v_1',1),\ldots,v_{r-1}=(v_{r-1}',1), v_r=(\mathbf{0},1),
\end{equation}
where $\mathbf{0}$ is the zero of $N$.
We introduce the convention of replacing $x_r$ by the variable $p$, to emphasize the distinguished role of $v_r$.

\medskip\noindent
{\bf Example.}
Local $\mathbb{P}^2$ can be described by a fan $\Sigma$ with edges generated by the vectors
\begin{equation*}
v_1=(1,0,1),\ v_2=(0,1,1),\ v_3=(-1,-1,1),\ v_4=(0,0,1).
\end{equation*}
We have maximal cones $\sigma_j$ for $j=1,2,3$.  Each $\sigma_j$ is spanned by the set of vectors $\{v_1,\ldots,v_4\}-\{v_j\}$.

In this case, the lattice $\Lambda$ is spanned by $(1,1,1-3)$, corresponding to the relation $v_1+v_2+v_3-3v_4=0$.  We get the generating vector field $\chi=x_1(\partial/\partial x_1)+x_2(\partial/\partial x_2)+x_3(\partial/\partial x_3)-3p(\partial/\partial p)$.  Putting $\Omega=i_\chi\left(dx_1\wedge dx_2\wedge dx_3\wedge dp
\right)$, we get
\begin{equation*}
\Omega=x_1 dx_2\wedge dx_3\wedge dp - x_2 dx_1\wedge dx_3\wedge dp+x_3dx_1\wedge dx_2\wedge dp+3pdx_1\wedge dx_2\wedge dx_3.
\end{equation*}

By respectively setting $x_i=1$ for $i=1,2,3$, we obtain 
\begin{equation*}
\Omega|_{X_{\sigma_1}} = dx_2\wedge dx_3\wedge dp,\Omega|_{X_{\sigma_2}} = dx_3\wedge dx_1\wedge dp,\Omega|_{X_{\sigma_3}} = dx_1\wedge dx_2\wedge dp.
\end{equation*}
Accordingly, we can transfer the construction of Section~\ref{subsec:} from $\operatorname{Hilb}^n(\mathbb{C}^3)$ to each $\operatorname{Hilb}^n(X_{\sigma_i})$, using the respective coordinate systems $(x_2,x_3,p), (x_3,x_1,p),(x_1,x_2,p)$.\footnote{We have abused notation here.  The relationship between the three local coordinates and homogeneous coordinates is given by three separate applications of (\ref{eq:affinecoordinates}).  It follows that the coordinate changes will be nontrivial, even for the variables in different coordinate systems denoted by the same symbol.} For example we can describe $\operatorname{NHilb}^n(X_{\sigma_1})$ as a GIT quotient of a space of $n\times n$ matrices $X_2,X_3,P$ together with a vector $v$, and define $\operatorname{Hilb}^n(X_{\sigma_1})$ as the degeneracy locus of $W_1=\operatorname{Tr}([X_2,X_3]P)$.  Similarly, we can define $\operatorname{Hilb}^n(X_{\sigma_2})$  as the degeneracy locus of $W_2=\operatorname{Tr}([X_3,X_1]P)$ in $\operatorname{NHilb}^n(X_{\sigma_2})$ and $\operatorname{Hilb}^n(X_{\sigma_3})$  as the degeneracy locus of $W_3=\operatorname{Tr}([X_1,X_2]P)$ in $\operatorname{NHilb}^n(X_{\sigma_3})$.  This description gives the data $(\operatorname{Hilb}^n(X_{\sigma_j}), \operatorname{NHilb}^n(X_{\sigma_j}), W_j, i_j)$, where $i_j:\operatorname{Hilb}^n(X_{\sigma_j})\hookrightarrow \operatorname{NHilb}^n(X_{\sigma_j})$ is the inclusion.

We will show that the $(\operatorname{Hilb}^n(X_{\sigma_j}), \operatorname{NHilb}^n(X_{\sigma_j}), W_j, i_j)$ are among the d-critical charts of a d-critical locus on $\operatorname{Hilb}^n(\omega_{\mathbb{P}^2})$. The construction of the d-critical locus structure proceeds in two steps.  

First, we show that the sections $s_j$ of $S^0_{^n(X_{\sigma_j})}$ induced by the $W_j$ agree on pairwise intersections $\operatorname{Hilb}^n(X_{\sigma_j})\cap \operatorname{Hilb}^n(X_{\sigma_k})\subset \operatorname{Hilb}^n(\omega_{\mathbb{P}^2})$.  We will show this more generally for any local toric Calabi-Yau threefold in Proposition~\ref{prop:opendcrit} below.  But this is still not enough to construct a d-critical locus, since the $\operatorname{Hilb}^n(X_{\sigma_j})$ do not cover $\operatorname{Hilb}^n(\omega_{\mathbb{P}^2})$.

Second, we construct a cover of $\omega_{\mathbb{P}^2}$ by open subsets $U_\alpha\simeq\mathbb{C}^3$ such that the $\operatorname{Hilb}^n(U_\alpha)$ cover $\operatorname{Hilb}^n(\omega_{\mathbb{P}^2})$.  Furthermore, automorphisms of $\mathbb{P}^2$ induce isomorphisms $U_\alpha\simeq X_{\sigma_j}$.  These isomorphisms are used to construct the data $(\operatorname{Hilb}^n(U_\alpha), \operatorname{NHilb}^n(U_\alpha), W_\alpha, i_\alpha)$.  Finally, a comparison to the first step shows that the sections $s_\alpha$ of $S^0_{^n(U_\alpha)}$ are pairwise compatible.  We will elaborate on this step more generally in the proof of Theorem~\ref{thm:dcrit} below.

\medskip
Returning to the general case, we have the following easy lemma.  The coordinates $(x,y,z)$ on $X_\sigma$ in the statement of the lemma are given by (\ref{eq:affinecoordinates}).  

\begin{lemma}\label{lem:plus1}
Let $\Omega$ be the holomorphic 3-form (\ref{eqn:omega}) on the local Calabi-Yau threefold $\omega_S$ and let $\sigma$ be a maximal cone in a fan for $\omega_S$.  Then we can order the coordinates$(x,y,z)$ on $X_\sigma\simeq\mathbb{C}^3$ so that $\Omega|_{X_\sigma}=dx\wedge dy\wedge dz$. 
\end{lemma}

We had already seen this for $\omega_{\mathbb{P}^2}$ by direct calculation.

\begin{proof}
 Reordering $\Sigma(1)$ if necessary, we can assume that the edges of $\sigma$ are generated by $v_{r-2},v_{r-1}$, and $v_r$. We have coordinates $(x_{r-2},x_{r-1},x_r)$ on $X_\sigma$ \ccgreen{as in (\ref{eq:affinecoordinates})}. Choosing generators $\lambda_1,\ldots,\lambda_{r-3}$ of the lattice (\ref{eqn:Lambda}) and writing $\lambda_i=(a_{i1},\ldots,a_{ir})$, we get $\Omega|_{X_\sigma}=\pm\operatorname{det}(A)\, dx_{r-2}\wedge dx_{r-1}\wedge dx_r$, where $A$ is the $(r-3)\times (r-3)$ matrix with entries $a_{ij}, 1\le i,j\le r-3$.  Choosing a different set of generators \ccgreen{for the lattice} in the definition of $\Omega$ can only change $\Omega|_{X_\sigma}$ by a sign as remarked earlier.  However, since $\{v_{r-2},v_{r-1},v_r\}$ is a basis for $N$, \ccgreen{for each $1\le i\le n$ we can uniquely write 
\begin{equation*}
v_i=c_{i,r-2}v_{r-2}+c_{i,r-1}v_{r-1}+c_{i,r}v_r.
\end{equation*}
Then choosing
\begin{equation}\label{eq:latticegen}
\lambda_i=(0,\ldots,0,1,0\ldots0,-c_{i,r-2},-c_{i,r-1},-c_{i,r}),\ i=1,\ldots,r-3
\end{equation}
as generators for the lattice (\ref{eqn:Lambda}), we see that $A$ is the indentity matrix, so that}
$\operatorname{det}(A)=1$.  \ccgreen{In (\ref{eq:latticegen}), the entry 1 is in the $i$\textsuperscript{th} position.}
\end{proof}

Now fixing a choice of $\Omega$ on $\omega_S$, for each maximal cone $\sigma$ we choose coordinates $(x,y,z)$ on $X_\sigma$ so that Lemma~\ref{lem:plus1} holds, noting that the coordinates $(x,y,z)$ can undergo a cyclic permutation and still satisfy Lemma~\ref{lem:plus1}.  We repeat the description of $\operatorname{Hilb}^n(\mathbb{C}^3)$ in this context, introducing endomorphisms $X_\sigma,Y_\sigma,Z_\sigma\in Hom(V_n,V_n)$ and putting $\operatorname{NHilb}(X_\sigma)=Hom(V_n,V_n)^3\times V_n\sslash GL_n$ as before. This determines the data $(\operatorname{Hilb}^n(X_{\sigma}), \operatorname{NHilb}^n(X_{\sigma}), W_\sigma, i_\sigma)$, where $W_\sigma$ is the function on $\operatorname{NHilb}(X_\sigma)$  given by $\operatorname{tr}([X_\sigma,Y_\sigma]Z_\sigma)$.   Since $W_\sigma$ is unchanged by a cyclic permutation of $(X_\sigma,Y_\sigma,Z_\sigma)$, we conclude that $(\operatorname{Hilb}^n(X_{\sigma}), \operatorname{NHilb}^n(X_{\sigma}), W_\sigma, i_\sigma)$ only depends on $\sigma$.

\ccgreen{Put $R_\sigma=\operatorname{Hilb}^n(X_\sigma)$ and $U_\sigma=\operatorname{NHilb}^n(X_\sigma)$. Using the gluing data of the $X_\sigma$'s, we see that the kernels of the maps
	\begin{equation*}
		\frac{i^{-1}(O_{R_\sigma})}{I^2_{R_\sigma, U_\sigma}}\rightarrow \frac{i^{-1}(T^*R_\sigma)}{I_{R_\sigma, U_\sigma}\cdot i^{-1}(T^*R_\sigma)}
		\end{equation*}
	in Theorem \ref{thm1} glue to a global sheaf $\mathcal{T}$ on $\cup_{\sigma}^n(X_\sigma)$. This sheaf is isomorphic to the sheaf $\mathcal{S}_{\cup_{\sigma}^n(X_\sigma)}$ by the construction of the sheaf $\mathcal{S}$ in \cite{Joy15}.}

\ccgreen{Let $s_\sigma\in \mathcal{S}^0_{\operatorname{Hilb}(X_\sigma)}$ be the section induced by $W_\sigma$.}

\begin{proposition}\label{prop:opendcrit}
\begin{sloppypar}
Given maximal cones $\sigma$ and $\sigma'$, the sections $s_\sigma$ and $s_{\sigma'}$ agree on $\operatorname{Hilb}^n(X_\sigma)\cap \operatorname{Hilb}^n(X_{\sigma'})$, \ccgreen{understood as a subset of $\operatorname{Hilb}^n(\omega_S)$}.
\end{sloppypar}
\end{proposition}

We therefore have a canonical d-critical locus structure on
 $\cup_{\sigma}\operatorname{Hilb}^n(X_\sigma)\subset \operatorname{Hilb}^n(\omega_S)$ with critical charts $(\operatorname{Hilb}^n(X_{\sigma}), \operatorname{NHilb}^n(X_{\sigma}), W_\sigma, i_\sigma)$.

\begin{proof}
We begin by observing that the proposition can be reduced to the case when $\sigma$ and $\sigma'$ share a codimension~1 face.
If $\sigma$ and $\sigma'$ do not share a codimension~1 face, then
we can choose maximal cones $\sigma_1,\sigma_2,\ldots,\sigma_m$ with $m\ge 3$ and $\sigma_1=\sigma$ and $\sigma_m=\sigma'$ such that each pair $\sigma_i,\sigma_{i+1}$ has a common codimension~1 face.  Furthermore, $X_\sigma\cap X_{\sigma'}=X_{\rho}\simeq \mathbb{C}\times(\mathbb{C}^*)^2$, where $\rho$ is the 1-dimensional cone generated by $v_r=(0,0,1)$ in the notation (\ref{eq:edges}).  Since $\rho$ is a face of every $\sigma_i$, we have $X_\rho\subset X_{\sigma_i}$ for each $i$.  Correspondingly, $\operatorname{Hilb}^n(X_\sigma)\cap \operatorname{Hilb}^n(X_{\sigma'})=\operatorname{Hilb}^n(X_\rho)\subset \operatorname{Hilb}^n(X_{\sigma_i})$ for all $i$. \ccred{Here the intersection of the Hilbert schemes are taken inside $\operatorname{Hilb}^n(\omega_n)$.} Once we have proven the proposition for cones sharing a codimension~1 face, it will follow immediately that on $\operatorname{Hilb}^n(X_\rho)$ we have $s_\sigma=s_{\sigma_2}=s_{\sigma_3}=\ldots=s_{\sigma'}$ as required.

The common codimension~1 face of $\sigma$ and $\sigma'$ is spanned by $v_r$ and some $v_i=(v_i',1)$ for $i<r$.  The vector $v_i'\in N'$ corresponds to a torus-invariant curve $D'_i$ in the toric surface $S$.  We let $\eta$ be the self-intersection of $D'_i$ in $S$, so that $v'_{i-1}+\eta v'_i+v'_{i+1}=0$.\footnote{If $i=1$ or $i=r-1$, we cyclically reorder the $v'_j$ so that $v'_{i+1}$ makes sense.}
It follows that $v_{i-1}+\eta v_i+v_{i+1}=(\eta+2)v_r$.  Using the basis $\{v_{i-1},v_i,v_r\}$ for $N$ to identify $N$ with $\mathbb{Z}^3$, we have, reordering $\sigma$ and $\sigma'$ if necessary
\begin{equation*}
\sigma=\mathrm{span}\{(1,0,0),(0,1,0),(0,0,1)\},\ \sigma'=\mathrm{span}\{(0,1,0),(-1,-\eta,\eta+2),(0,0,1)\}.
\end{equation*}

The coordinate change between $X_\sigma$ and $X_{\sigma'}$ is given by 
\begin{equation}\label{eq:change}
(x',y',z')=(x^{-\eta}y,x^{-1},x^{\eta+2}z),
\end{equation}
\ccgreen{as is easily computed.}
We have ordered the coordinates so that $dx'\wedge dy'\wedge dz'=dx\wedge dy \wedge dz$.  If $\Omega|_{X_\sigma}=-dx\wedge dy\wedge dz$, we simply reorder the coordinates appropriately without impacting the rest of the argument.

We next transfer the construction of Section~\ref{subsec:} to $X_{\sigma}$ and to $X_{\sigma'}$ while providing some more detail.
We let $X,Y$ and $Z$ be three $n\times n$ matrices with indeterminate entries $X(k, l)$, $Y(k, l)$ and $Z(k, l)$, representing multiplication by $x$, $y$ and $z$, respectively. Let $u_1,..., u_n$ be $n$ additional indeterminates and put 
\begin{equation*}
A=\mathbb{C}[X(k, l), Y(k, l), Z(k, l), u_1,...,u_n].
\end{equation*} 
Denote by $U$ the open subset $(\operatorname{Spec}(A))^{ss}$ of $\operatorname{Spec}(A)$ with respect to the natural $GL_n$ \ccgreen{action, the semistable locus in $\operatorname{Spec}(A)$ with respect to $\chi$ as in Section 2.2}. Put $N=U/GL(n)$.  Put $d=\operatorname{det}(X)$ and let $A_d=A[d^{-1}]$. 
We denote $U\cap SpecA_{d}$ by $U_{d}$. We similarly define 
\begin{equation*}
A'=\mathbb{C}[X'(k,l),Y'(k,l),Z'(k,l),u'_1,\ldots,u'_n],
\end{equation*}
$U'=(\operatorname{Spec}(A'))^{ss}$, $N'=U'/GL_n$, $A'_{d'}=A'[\operatorname{det}({X'})^{-1}]$, and $U'_{d'}=U'\cap SpecA_{d'}$.

\ccgreen{Noting that the entries of $X^{-1}$ are in $A_d$ and the entries of $(Y')^{-1}$ are in $A'_{d'}$, we can construct} an isomorphism 
\begin{equation}
\label{equ3.4}
\phi: A'_{d'}\simeq A_{d}
\end{equation}
which identifies
\begin{equation*}
X'=X^{-\eta}Y,
Y'=X^{-1}, 
Z'=X^{\eta+2}Z,
u'_i=u_i,\ i=1,\ldots n
\end{equation*}
induced by the change of coordinates (\ref{eq:change}) after making a choice in the order of matrix multiplication. This induces the isomorphism of open subsets:
\begin{equation}
\label{equ3.5}
U_{d}\cap \phi^{-1}(U_{d'})\simeq \phi(U_{d})\cap U'_{d'},
\end{equation}
where we have abused notation slightly by using $\phi$ to also denote the induced map $\operatorname{Spec}(A_{d})\to \operatorname{Spec}(A'_{d'})$.

To simplify notation, we put $R=\operatorname{Hilb}^n(X_\sigma)$ and $R'=\operatorname{Hilb}^n(X_{\sigma'})$.  Let $i:R\to N$ and $i':R'\to N'$ be the inclusions.  Let $R\cap R'$ be the intersection of $R$ and $R'$ taken inside $\operatorname{Hilb}^n(\omega_S)$. Note that $U_{d}$ and $U'_{d'}$ are invariant under the $GL_n$ actions. Hence $(U_{d}\cap \phi^{-1}(U'_{d'}))/GL_n$ is an open neighborhood of $R\cap R'$ in $U_{d}/GL_n\subset N$, and $(\phi(U_{d})\cap U')/GL_n$ is an open neighborhood of $R\cap R'$ in $U'_{d"}/GL_n\subset N'$. Then (\ref{equ3.5}) induces an isomorphism
\begin{equation}
\label{equ3.6}
i^{-1}(O_{N})/I_{R, N}^2|_{R\cap R'}\simeq {i'}^{-1}(O_{N'})/I_{R', N'}^2|_{R\cap R'},
\end{equation}
\ccred{where $i^{-1}(O_N)$ and $I_{R, N}$ are as defined in Theorem \ref{thm1}.}

We have the exact sequence of sheaves on $N$
\begin{equation*}
0\to K\to i^{-1}(O_{N})/I_{R, N}^2\xrightarrow{d} i^{-1}(T^*N)/I_{R, N}\cdot i^{-1}(T^*N)
\end{equation*}
where $K$ is defined as the kernel, and is naturally identified with the sheaf $\mathcal{S}_{R}$.  We have a similar exact sequence defining a sheaf $K'$ on $R'$.

We deduce the following commutative diagram, where the vertical arrows are isomorphisms 
\begin{equation}
\label{eq:isomk}
\begin{array}{ccccccc}
0&\to& \left(K\right)|_{R\cap R'}&\to& \left(i^{-1}(O_{N})/I_{R, N}^2\right)|_{R\cap R'}&\xrightarrow{d}& \left(i^{-1}(T^*N)/I_{R, N}\cdot i^{-1}(T^*N)\right)|_{R\cap \ccgreen{R'}}\\
&&\downarrow&&\downarrow&&\downarrow\\
0&\to& \left(K'\right)|_{R\cap R'}&\to& \left({i'}^{-1}(O_{N'})/I_{R', N'}^2\right)|_{R\cap R'}&\xrightarrow{d}& \left({i'}^{-1}(T^*N')/I_{R', N'}\cdot {i'}^{-1}(T^*N')\right)|_{R\cap R'}\ccgreen{.}
\end{array}
\end{equation}
We show that the sections $s_\sigma\in H^0(K)$ and $s_{\sigma'}\in H^0(K')$ agree on $R\cap R'$ according to the leftmost isomorphism in (\ref{eq:isomk}).  For this, it suffices to show that their images in $i^{-1}(O_{N})/I_{R, N}^2$ and ${i'}^{-1}(O_{N'})/I_{R', N'}^2$ agree on $R\cap R'$ via the isomorphism (\ref{equ3.6}).  We proceed to verify this equality.

We have the potentials $W_\sigma=\operatorname{tr}([X,Y]Z])$ on $U$ and $W_{\sigma'}=\operatorname{tr}([X',Y']Z])$ on $U'$. 
Denote the image of $i^{-1}(W_\sigma)$ in the sheaf $i^{-1}(O_{N})/I_{R, N}^2$ by $i^{-1}(W_\sigma)+I_{R, N}^2$, with analagous notation for ${i'}^{-1}(W_{\sigma'})$. Then we claim that:

\begin{equation}\label{eq:compatible}
i^{-1}(W_\sigma)+I_{R, N}^2|_{R\cap R'}={i'}^{-1}(W_{\sigma'})+I_{R', N'}^2|_{R\cap R'}\rm{\  under\ the\ identification\ (\ref{equ3.6})}.
\end{equation} 
This is precisely the assertion of the proposition that the sections $s_\sigma$ and $s_{\sigma'}$ agree on $R\cap R'$.

We now demonstrate (\ref{eq:compatible}).
Under the isomorphism (\ref{equ3.4}), we have 
\begin{equation*}
W_\sigma|_{\operatorname{Spec}(A_{d})}=\operatorname{tr}(XYZ-YXZ)
\end{equation*}
\begin{equation}\label{eq:w'}
W_{\sigma'}|_{\operatorname{Spec}(A'_{d'})}=\operatorname{tr}(X^{-\eta}YX^{\eta+1}Z-X^{-(\eta+1)}YX^{\eta+2}Z).
\end{equation}
Making use of the identity \ccgreen{$\operatorname{tr}(BC)=\operatorname{tr}(CB)$ for $n\times n$ matrices $B$ and $C$}, we can rewrite (\ref{eq:w'}) as
\begin{equation}\label{eq:w'2}
W_{\sigma'}|_{\operatorname{Spec}(A'_{d'})}=\operatorname{tr}(X^{\eta+1}ZX^{-\eta}Y-X^{\eta+2}ZX^{-(\eta+1)}Y).
\end{equation}
Since (\ref{eq:w'2}) can be obtained from (\ref{eq:w'}) by exchanging \cgreen{$Y$ and $Z$}, replacing $\eta$ by $-\eta-2$, and multiplying by $-1$, we can and will assume that $\eta\ge -1$.  If $\eta=-1$, then $W_\sigma|_{\operatorname{Spec}(A_{d})}=W_{\sigma'}|_{\operatorname{Spec}(A'_{d'})}$.  If $\eta=0$ we compute
\begin{equation}\label{eq:f0}
W_\sigma|_{\operatorname{Spec}(A_{d})}-W_{\sigma'}|_{\operatorname{Spec}(A'_{d'})}=\operatorname{tr}\left(XYZ-2YXZ+X^{-1}YX^2Z\right)=\operatorname{tr}\left(X^{-1}[X,Y][Z,X]\right),
\end{equation}
which is visibly in $I_{R, N}^2|_{R\cap R'}$.

We proceed by induction on \ccgreen{$\eta$}, assuming the inductive hypothesis
\begin{equation}\label{eq:f}
\operatorname{tr}\left(\left(X^{-\eta}YX^{\eta+1}Z-X^{-(\eta+1)}YX^{\eta+2}Z\right)-\left(XYZ-YXZ\right)\right)\in I_{R, N}^2|_{R\cap R'}.
\end{equation}
The substitution $(X,Y,Z)\mapsto (X,YX,Z)$ preserves $I_{R, N}$ and hence $I_{R,N}^2$, since
\begin{equation*}
[X,YX]=\ccgreen{[X,Y]X},\ [Z,YX]=[Z,Y]X+Y[Z,X].
\end{equation*}
Making this substitution in (\ref{eq:f}) we get
\begin{equation}\label{eq:induct}
\operatorname{tr}\left(\left(X^{-\eta}YX^{\eta+2}Z-X^{-(\eta+1)}YX^{\eta+3}Z\right)-\left(XYXZ-YX^2Z\right)\right)\in I_{R, N}^2|_{R\cap R'}.
\end{equation}
We write
\begin{equation*}
X^{-(\eta+1)}YX^{\eta+2}Z-X^{-(\eta+2)}YX^{\eta+3}Z=X^{-1}\left(X^{-\eta}YX^{\eta+2}Z-X^{-(\eta+1)}YX^{\eta+3}Z
\right),
\end{equation*}
so that we have by (\ref{eq:induct}) modulo $I_{R, N}^2|_{R\cap R'}$
\begin{equation*}
\operatorname{tr}(X^{-(\eta+1)}YX^{\eta+2}Z-X^{-(\eta+2)}YX^{\eta+3}Z)\equiv \operatorname{tr}(YXZ-X^{-1}YX^2Z)
\end{equation*}
and therefore 
\begin{equation}\label{eq:almostplus1}
\operatorname{tr}\left(\left(X^{(-\eta+1)}YX^{\eta+2}Z-X^{-(\eta+2)}YX^{\eta+3}Z\right)-\left(YXZ-X^{-1}YX^2Z\right)\right)\in I_{R, N}^2|_{R\cap R'}.
\end{equation}
We also have
\begin{equation*}
\operatorname{tr}\left(XYZ-2YXZ+X^{-1}YX^2Z\right)\in I_{R, N}^2|_{R\cap R'}.
\end{equation*}
by (\ref{eq:f0}).   Subtracting this last equation from (\ref{eq:almostplus1}), we obtain (\ref{eq:f}) with $\eta$ replaced by $\eta+1$, proving the inductive step, and we are done.
\end{proof}

We now come to the main result of this section.

\begin{theorem}\label{thm:dcrit}
Suppose that $S=\mathbb{P}^2$ or $S=\mathbb{F}_n$.  Then $\operatorname{Hilb}^n(\omega_S)$ has a natural d-critical locus structure. 
The critical charts are all isomorphic to $(\operatorname{Hilb}^n(\mathbb{C}^3),\operatorname{NHilb}^n(\mathbb{C}^3),W,i)$.
\end{theorem}

\begin{proof}

Let $\pi:\omega_S\to S$ be the projection.  Our strategy is to exhibit an open cover $\{U_\alpha\}$ of $S$ satisfying the two properties:
\begin{enumerate}
\item There exists an automorphism of $\omega_S$ which preserves $\Omega$ and takes $V_\alpha:=\pi^{-1}(U_\alpha)$ to $X_{\sigma_\alpha}$ for some maximal cone $\sigma_\alpha$.
\item Any finite subset of $S$ is contained in some $U_\alpha$.
\end{enumerate}
The first condition allows us to define charts
\begin{equation*}
(\operatorname{Hilb}^n(V_\alpha),\operatorname{NHilb}^n(V_\alpha),W_\alpha,i_\alpha)
\end{equation*}
\ccgreen{by transferring the chart $(\operatorname{Hilb}^n(X_{\sigma_\alpha})),\operatorname{NHilb}^n(X_{\sigma_\alpha}),W_{\sigma_\alpha},i_{\sigma_\alpha})$ from $\operatorname{Hilb}^n(X_{\sigma_\alpha})$ to $\operatorname{Hilb}^n(V_\alpha)$ via this} automorphism.   The second condition implies that $\{\operatorname{Hilb}^n(V_\alpha)\}$ is an open cover of $\operatorname{Hilb}^n(\omega_S)$. To see this, let $Q\in \operatorname{Hilb}^n(\omega_S)$.  Thinking of $Q$ as a length $n$ subscheme of $\omega_S$, we choose a $U_\alpha$ containing the finite set $\pi(Q)$.  Then $Q\subset V_\alpha$, so that $Q\in \operatorname{Hilb}^n(V_\alpha)$.

The rest of the proof consists of choosing the $U_\alpha$ and automorphisms so that the sections $s_\alpha\in S^0_{\operatorname{Hilb}^n(V_\alpha)}$ agree on pairwise intersections.
We work out the details for $S=\mathbb{P}^2$ and $S=\mathbb{F}_n$ separately.

\smallskip\noindent
$S=\mathbb{P}^2$:  For a line $\ell\subset \mathbb{P}^2$, we let $U_\ell=\mathbb{P}^2-\ell$ be its complement.  Then $\{U_\ell\}$ is an open cover of $\mathbb{P}^2$ which clearly satisfies the second property.   To check the first property, fix an equation for $\ell$ and let $A$ be a determinant~1 linear transformation of $\mathrm{span}(x_1,x_2,x_3)$ which induces an automorphism of $\mathbb{P}^2$  taking $\ell$ to one of the torus invariant lines.
In terms of the homogeneous coordinates $(x_1,x_2,x_3,p)$ for $\omega_{\mathbb{P}^2}$, the automorphisms we consider are  $\phi_A(\mathbf{x},p)=(A\mathbf{x},p)$ where $\mathbf{x}=(x_1,x_2,x_3)$.    
 Since both $\widehat{\Omega}=dx_1\wedge dx_2\wedge dx_3\wedge dp$ and $v$ are invariant under $A$, it follows that $\Omega$ is invariant under the automorphisms $\phi_A$.

Given any pair of lines $\ell,\ell'$, we can find \ccgreen{a determinant $1$ linear transformation $A$ taking} $\ell$ and $\ell'$ to $x_1$ and $x_2$ respectively. Then $\phi_A$ identifies
$\operatorname{Hilb}^n(V_\ell)\cap \operatorname{Hilb}^n(V_{\ell'})$ with $\operatorname{Hilb}^n(X_{\sigma_1})\cap \operatorname{Hilb}^n(X_{\sigma_2})$ using the labeling of the cones $\sigma_i$ introduced in the example earlier in this section. \ccgreen{As in the discussion before Proposition \ref{prop:opendcrit}, using the gluing data between $\operatorname{Hilb}^n(U_\ell)$ and $\operatorname{Hilb}^n(U_{\ell'})$ induced by the change of coordinates between $U_\ell$ and $U_{\ell'}$, the local sheaves $\mathcal{S}_{\operatorname{Hilb}^n(U_\ell)}$ glue to the sheaf $\mathcal{S}_{\operatorname{Hilb}(\omega_{\mathbb{P}^2})}$. Then by Remark \ref{lineartrans},} the compatibility of the \ccred{local sections associated to the }critical charts of $\operatorname{Hilb}^n(V_\ell)$ and $\operatorname{Hilb}^n(V_{\ell'})$ follows immediately from the compatibility of the critical charts of $\operatorname{Hilb}^n(X_{\sigma_1})$ and $\operatorname{Hilb}^n(X_{\sigma_2})$ proven in Proposition~\ref{prop:opendcrit}.

\smallskip\noindent
$S=\mathbb{F}_n$:  We describe $\mathbb{F}_n$ as the toric surface associated to the complete fan with 1-dimensional cones spanned by the vectors
\begin{equation}
v'_1=(1,0),\ v'_2=(0,1),\ v'_3=(-1,n),\ v'_4=(0,-1).
\end{equation}
The torus-invariant curves $D'_1$ and $D'_3$ are fibers of $\mathbb{F}_n$, the curve $D'_2$ is the section with $(D'_2)^2=-n$, and $D'_4$ is a section with $(D'_4)^2=n$.  The lattice of vectors fields associated to the quotient construction of $\omega_{\mathbb{F}_n}$ is generated by
\begin{equation*}
\chi_1\!=\!x_2(\partial/\partial x_2)\! +\! x_4(\partial/\partial x_4)\! -\! 2p(\partial/\partial p), \chi_2\!=\!x_1(\partial/\partial x_1)\! +\! x_3(\partial/\partial x_3)\! +\! nx_4(\partial/\partial x_4)\! -\! (n+2)p(\partial/\partial p).
\end{equation*}

The automorphisms of $\omega_{\mathbb{F}_n}$ which we use are of two types, described in terms of homogeneous coordinates by
\begin{enumerate}
\item $(x_1,x_2,x_3,x_4,p)\mapsto (ax_1+bx_3,x_2,cx_1+dx_3,x_4,p)$, where $ad-bc=1$
\item $(x_1,x_2,x_3,x_4,p)\mapsto (x_1,x_2,x_3,x_4+ex_2x_3^n,p)$.
\end{enumerate}

These automorphisms preserve $\widehat{\Omega}$ and the $\chi_i$, hence they preserve $\Omega$ as well.

Let $\phi:\mathbb{F}_n\to \mathbb{P}^1$ be the map exhibiting $\mathbb{F}_n$ as a $\mathbb{P}^1$-bundle over $\mathbb{P}^1$.   In homogeneous coordinates, we have $\phi(x_1,x_2,x_3,x_4)=(x_1,x_3)$.   

Given a point $p=(a,b)\in \mathbb{P}^1$,
we have a 1-parameter family of sections of $S_{a,b,e}$ of $\mathbb{F}_n$ described parametrically  as $\{(x_1,1,x_3,e(bx_1-ax_3)^n)\mid (x_1,x_3)\in \mathbb{P}^1\}$.
Intrinsically, these sections are the irreducible members of the linear system $|D'_4|$ which intersect $D'_4$ only in the point $(a,1,b,0)$ of $D'_4$, with multiplicity $n$.  Since $S_{\rho a,\rho b,\rho^{-n}e}=S_{a,b,e}$ \ccred{for any $\rho\in \mathbb{C}^*$}, we see that this 1-parameter family of sections only depends on $p$ and not on a choice of homogenous coordinates $(a,b)$ of $p$.     We use the cover $\{U_{a,b,e}\}=\{\phi^{-1}(\mathbb{P}^1-(a,b))-S_{a,b,e}\}$ of $\mathbb{F}_n$.  Let $V_{a,b,e}=\pi^{-1}(U_{a,b,e})$. \ccgreen{Similar to the case of $\mathbb{P}^2$, the local sheaves $\mathcal{S}_{\operatorname{Hilb}^n(U_{a, b, e})}$ glue to a global sheaf on $\operatorname{Hilb}^n(\omega_{\mathbb{F}_n})$ via the change of coordinates between the $U_{a, b, e}\subset \mathbb{F}_n$.}

To verify the first property required of $\{U_\alpha\}$, we can find an automorphism of the first type mapping $V_{a,b,e}$ to $V_{1,0,e'}$.  Then an automorphism of the second type maps $V_{1,0,e'}$ to $V_{1,0,0}$.  But $V_{1,0,0}$ is equal to $X_\sigma$, where $\sigma$ is the maximal cone generated by $\{v_1=(1,0,1),v_2=(0,1,1),$ \ccred{$v_5=(0,0,1)\}$}.   We use these automorphisms to identify the function $W_\sigma$ on $\operatorname{NHilb}^n(X_\sigma)$ with a function $W_{a,b,e}$ on $\operatorname{NHilb}^n(V_{a,b,e})$.  Let $s_{a,b,e}$ be the section of $S^0_{^n(V_{a,b,e})}$ induced by $W_{a,b,e}$.


Given a finite subset $T$ of $\mathbb{F}_n$, we first choose $p=(a,b)\in \mathbb{P}^1-\phi(T)$.  Then we can find an $e$ such that $S_{a,b,e}\cap T$ is empty, so that $T\subset U_{a,b,e}$.  This verifies the second property required of $\{U_\alpha\}$.

It remains to show pairwise compatibility of the sections $s_{a,b,e}$ of $S^0_{\operatorname{Hilb}^n(V_{a,b,e})}$.  Consider distinct open subsets $V_{a,b,e}$ and $V_{a',b',e'}$.  There are two cases to consider:
\begin{enumerate}
\item $(a,b)= (a',b')$ as points of $\mathbb{P}^1$, or
\item $(a,b)\ne(a',b')$ as points of $\mathbb{P}^1$.
\end{enumerate}
In the first case, we can find a determinant 1 linear transformation of $\mathrm{span}(x_1,x_3)$ taking $(a,b)$ to $(1,0)$.  Then we use an automorphism of the second type to take $V_{a,b,e}$ to $V_{1,0,0}$.  This automorphism takes $V_{a',b',e'}$ to $V_{1,0,e''}$ for some $e''\ne0$.  We are reduced to comparing $V_{1,0,0}$ and $V_{1,0,e''}$.  On $V_{1,0,0}$ we have affine coordinates $(x_1,x_2,p)$ following (\ref{eq:affinecoordinates}).  Our construction gives us affine coordinates $(x_1',x_2',p')$ on $U_{1,0,e''}$.  We relate $(x_1',x_2',p')$ to $(x_1,x_2,p)$ using full homogenous coordinates.  We have
\begin{equation*}
(x_1',x_2',1,1,p)=(x_1,x_2,1,1+e''x_2,p)=(x_1,x_2(1+e''x_2)^{-1},1,1,(1+e''x_2)^2p),
\end{equation*}
where we have used the $(\mathbb{C}^*)^2$ of the quotient construction in the last step.  This gives the coordinate change as
$(x_1',x_2',p')=(x_1,x_2(1+e''x_2)^{-1},(1+e''x_2)^2p)$.  Promoting $x_1,x_2,p,x_1',x_2',p'$ to elements of $End(V_n)$, we compute

\begin{equation*}
[X_1',X_2']P'-[X_1,X_2]P=[X_1,X_2(1+e''X_2)^{-1}](1+e''X_2)^2P-[X_1,X_2]P,
\end{equation*}
which simplifies to 
\begin{equation*}
e''X_1X_2^2P+X_2X_1P-X_2(1+e''X_2)^{-1}X_1(1+e''X_2)^2P.
\end{equation*}
\ccgreen{Using the cyclic property of the trace and freely commuting terms involving $X_2$ only, we see that this last expression has the same trace as that of
\begin{equation*}
[X_1(1+e''X_2),X_2(1+e''X_2)^{-1}]\ [1+e''X_2,P].
\end{equation*}}
We only have to show that this trace is in the square of the ideal $I$  generated by the entries of the commutators of $X_1,X_2$, and $P$, appropriately localized so that $1+e''X_2$ is invertible.  We only have to show that the entries of $[X_1(1+e''X_2),X_2(1+e''X_2)^{-1}]$ are in $I$, or equivalently that $X_1(1+e''X_2)X_2(1+e''X_2)^{-1}$ is congruent to $X_2(1+e''X_2)^{-1}X_1(1+e''X_2)$ modulo $I$.  We can do this by commuting adjacent terms successively.  For example, to show that  $(1+e''X_2)^{-1}X_1$ is congruent to $X_1(1+e''X_2)^{-1}$ moduli $I$, we only have to observe that $X_1(1+e''X_2)$ is congruent to $(1+e''X_2)X_1$ modulo $I$, and then multiply by $(1+e''X_2)^{-1}$ on the left and on the right.

We abbreviate our notation by putting $H_{a,b,e}=\operatorname{Hilb}^n(V_{a,b,e})$ and $NH_{a,b,e}=\operatorname{NHilb}^n(V_{a,b,e})$. 
Thus $W_{1,0,0}=W_{1,0,e''}$ on $H_{1,0,0}\cap H_{1,0,e''}$, hence $s_{1,0,0}=s_{1,0,e''}$ on $H_{1,0,0}\cap H_{1,0,e''}$ as well.

In the second case, we can find a determinant 1 linear transformation of $\mathrm{span}(x_1,x_3)$ taking $(a,b)$ to $(1,0)$ 
and $(a',b')$ to $(0,1)$.  Then we use an automorphism of the second type to take $V_{a,b,e}$ to $V_{1,0,0}$.  This automorphism takes $V_{a',b',e'}$ to $V_{0,1,e''}$ for some $e''$.  We are reduced to comparing $s_{1,0,0}$ and $s_{0,1,e''}$ on $H_{1,0,0}\cap H_{0,1,e''}$.  

If $e''=0$, we are already done by Proposition~\ref{prop:opendcrit}. For $e''\ne0$, we first note that $H_{1,0,0}\cap H_{0,1,e''}= H_{1,0,0}\cap H_{0,1,0}\cap H_{0,1,e''}$.  It follows that the three sections $s_{1,0,0}$ on $H_{1,0,0}$, $s_{0,1,0}$ on $H_{0,1,0}$, and $s_{0,1,e''}$ on $H_{0,1,e''}$ are all defined on $H_{1,0,0}\cap H_{0,1,e'}$  and so can be compared there.  The first two agree on $H_{1,0,0}\cap H_{0,1,0}$ by Proposition~\ref{prop:opendcrit} as just noted, and for the last two, they agree on $H_{0,1,0}\cap H_{0,1,e''}$ by a straightforward calculation analogous to what was just done above.  It follows that $s_{1,0,0}$ and $s_{0,1,e''}$ agree on $H_{1,0,0}\cap H_{0,1,e''}$ and we are done.

%
\end{proof}

\smallskip\noindent
{\bf Remark.}  Our methods are easily adapted to construct a d-critical locus structure on the Hilbert schemes of local $\mathbb{P}^1$.

\section{Orientation on $(\operatorname{Hilb}^n(\omega_S), s)$}\label{sec:orientation}

In this section, we understand $S$ to mean either $S=\mathbb{P}^2$ or $S=\mathbb{F}_n$.  We use the more generic notations $U_\alpha, V_\alpha, W_\alpha$, and $s_\alpha$ instead of using $U_\ell$ etc.\ for $S=\mathbb{P}^2$ and $U_{a,b,e}$ etc.\ for $S=\mathbb{F}_n$.
We first work out the virtual canonical bundle associated to the d-critical locus $(\operatorname{Hilb}^n(\omega_S), s)$ \ccblue{following the construction in [Dav], see also \cite{Shi18}}.
\begin{proposition}
\label{prop8}
Let $F$ be the universal object on $\operatorname{Hilb}^n(\omega_S)$ and let $\pi_1: \operatorname{Hilb}^n(\omega_S)\times \omega_S\rightarrow \operatorname{Hilb}^n(\omega_S)$ be the projection. Then we have 
\begin{equation}
	\label{vircanon}
K_{\operatorname{Hilb}^n(\omega_S), s}\simeq (\operatorname{det}(\pi_{1*}F)^*)^2.
\end{equation}
\end{proposition}
\begin{proof}
We first prove condition (i) in Definition \ref{defK}. Let $(R, U, f, i) = (\operatorname{Hilb}^n(V_\alpha), N_\alpha, i_\alpha, W_\alpha)$ be one of the critical charts of $(\operatorname{Hilb}^n(\omega_S), s)$ constructed in the proof of Theorem~\ref{thm:dcrit}. It may be more clear if we use the quiver description of $U = N_\alpha$. It is well known \cblue{(see e.g.\cite{BBS13})} that $\operatorname{NHilb}^n(\mathbb{C}^3)$ is isomorphic to the moduli space of stable representations \ccgreen{$(V_\infty,V_0)$} of the following quiver:
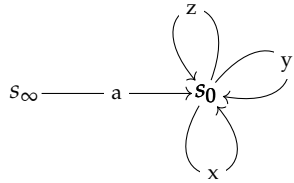
\begin{center}
\begin{tikzpicture}
    
\node [circle, minimum size=0.1pt,inner sep=0pt,outer sep=0pt] (a) at (0,0) {$s_\infty$};
\node[right=2cm of a, circle, minimum size=0.1pt,inner sep=0pt,outer sep=0pt] (b) {$s_0$} edge [in=-50,out=-120,loop, scale=3] node[midway, fill=white, scale=0.8]{x} ();
\node[right=2cm of a,circle,  minimum size=0.1pt,inner sep=0pt,outer sep=0pt] (b) {$s_0$} edge [in=-10,out=50,loop, scale=3] node[midway, fill=white, scale=0.8]{y} ();
\node[right=2cm of a, circle,  minimum size=0.1pt,inner sep=0pt,outer sep=0pt] (b) {$s_0$} edge [in=130,out=70,loop, scale=3] node[midway, fill=white, scale=0.8]{z} ();


\draw [->] (a.east) -- (b.west) node[midway, fill=white, scale=0.8]{a} ;

\end{tikzpicture}
\end{center}
with $V_\infty$ of dimension one and $V_0$ of dimension $n$. 
Here the quiver stability condition is equivalent to the condition that for any nonzero $v\in V_\infty$, $a(v)$ generates $V_0$ under the action of $x$, $y$ and $z$. 
\cblue{Let $W\subseteq \text{End}(V_0, V_0)^3\times \text{Hom}(V_\infty, V_0)$ be the locus of stable representations. Then we have \begin{equation*}
		N:=\operatorname{NHilb}^n(\mathbb{C}^3)\simeq W/\text{GL}(n).
	\end{equation*} 
We denote the projection from $W$ to $N$ by $\pi$.  Since $\text{GL}(n)$ acts equivariantly on $V_0$ and $V_\infty$ thought of as trivial bundles  on $W$, these bundles descend to bundles $E_0$ and $E_\infty$ on $N$ with $\pi^*(E_0)\simeq V_0$ and $\pi^*(E_\infty)\simeq V_\infty$ equivariantly.  We consider the short exact sequence
\begin{equation}
\label{cotangentseq}
0\rightarrow \pi^*T^*N\rightarrow T^*W\rightarrow T^*_{W/N}\rightarrow 0
\end{equation}
of bundles on $W$.  Note that $T^*_{W/N}$ is a trivial bundle with fiber $\mathfrak{gl}(n)^*$, on which $GL(n)$ acts by the dual of the adjoint action.  After modding out by $GL(n)$, (\ref{cotangentseq}) descends to an exact sequence of bundles on $N$
\begin{equation*}
0 \rightarrow T^*N \rightarrow G\rightarrow B \rightarrow 0.
\end{equation*}
Also,  $\operatorname{det}(B)$ is a trivial bundle, since the determinant of the adjoint action of $GL(n)$ is trivial.  
This implies that $\operatorname{det}(T^*N)\simeq\operatorname{det}(G)$.
Also \begin{equation}
\label{equcot}
\begin{split}
G&\simeq (E_0^*\otimes E_0)\oplus (E_0^*\otimes E_0)\oplus (E_0^*\otimes E_0)\oplus (E_0^*\otimes E_\infty).\\
\end{split}
\end{equation}
Considering $V_\alpha\simeq \mathbb{C}^3$, we have $(E_0)|_{\operatorname{Hilb}^n(V_\alpha)}\simeq \pi_{1*}F$ and the global section 1 of $\pi_{1*}F$ induces a trivialization of $(E_\infty)|_{\operatorname{Hilb}^n(V_\alpha)}$.
Hence we have a canonical isomorphism 
\begin{equation*}
K_{N_\alpha}^2|_{\operatorname{Hilb}^n(V_\alpha)}\simeq \operatorname{det}\left(\left(E_0^*\right)|_{\operatorname{Hilb}^n(V_\alpha)}\right)^2\simeq (\operatorname{det}(\pi_{1*}F)^*|_{\operatorname{Hilb}^n(V_\alpha)})^2.
\end{equation*}
This shows condition (i).  We denote the inverse of this isomorphism by $\iota_{(\operatorname{Hilb}^n(V_\alpha), N_\alpha, i_\alpha, W_\alpha)}$, with analogous notation for the other charts.}

\ccgreen{Now consider condition (ii) in Definition~\ref{defK}. We need to check that for any embeddings of critical charts \begin{equation*}
		\Phi: (\operatorname{Hilb}^n(V_\alpha)', N_\alpha', i_\alpha', W_\alpha')\rightarrow (\operatorname{Hilb}^n(V_\alpha), N_\alpha, i_\alpha, W_\alpha),
	\end{equation*} we have 
\begin{equation}
	\label{condii}
	\iota_{(\operatorname{Hilb}^n(V_\alpha), N_\alpha, i_\alpha, W_\alpha)}=i_{\alpha}'|_{\operatorname{Hilb}^n(V_\alpha)'}^*(J_\Phi)\circ \iota_{(\operatorname{Hilb}^n(V_\alpha)', N_\alpha', i_\alpha', W_\alpha')},
\end{equation}
where $N_\alpha'$ is a Zariski open subset of $N_\alpha$ and $\operatorname{Hilb}^n(V_\alpha)'=\operatorname{Hilb}^n(V_\alpha)\cap N_\alpha'$.
\cblue{
From the above, we see that $\iota_{(\operatorname{Hilb}^n(V_\alpha), N_\alpha, i_\alpha, W_\alpha)}$ is induced by $E_0|_{\operatorname{Hilb}^n(V_\alpha)}\simeq \pi_{1*}F|_{\operatorname{Hilb}^n(V_\alpha)}$ and $\iota_{(\operatorname{Hilb}^n(V_\alpha)', N_\alpha', i_\alpha', W_\alpha')}$ is induced by $E_0|_{\operatorname{Hilb}^n(V_\alpha)'}\simeq \pi_{1*}F|_{\operatorname{Hilb}^n(V_\alpha)'}$.
The Zariski open immersion $N'_\alpha\rightarrow N_\alpha$ is induced by a change of coordinates as in Theorem \ref{thm:dcrit}. The change of coordinates transforms the matrices associated with the edges $x, y, z$ accordingly, while leaving the universal bundles $E_0$ and $E_\infty$ unchanged. So we naturally obtain (\ref{condii}).}
}
\end{proof}

By Proposition \ref{prop8}, we see that $\operatorname{det}(\pi_{1*}F)^*$ \ccred{with the isomorphism \ref{vircanon}} defines an orientation on $(\operatorname{Hilb}^n(\omega_S), s)$. \ccred{We denote this choice of orientation by $K^{1/2}_{\operatorname{Hilb}^n(\omega), s}$}\ccgreen{. For later use, we also observe that from the calculations above we infer a canonical isomorphism 
\begin{equation}\label{eq:localor}
	\tau_{\operatorname{Hilb}^n(V_\alpha), N_\alpha, i_\alpha, W_\alpha}:K^{1/2}_{\operatorname{Hilb}^n(\omega), s}\rightarrow K_{N_\alpha}|_{\operatorname{Hilb}^n(V_\alpha)}.
\end{equation}
\cblue{induced by $\pi_{1*}F|_{\operatorname{Hilb}^n(V_\alpha)}\simeq E_0|_{\operatorname{Hilb}^n(V_\alpha)}$}
}

\section{Motivic DT invariants for $\ccgreen{\operatorname{Hilb}}^n(\omega_S)$}

In this section we use the d-critical locus structure from Section~\ref{sec:dcrit} and the orientation  from Section~\ref{sec:orientation} to compute the motivic DT invariants of $\ccgreen{\operatorname{Hilb}}^n(\omega_S)$, showing that it agrees with the computation in \cite{BBS13} for $S=\mathbb{P}^2$ and $S=\mathbb{F}_n$.

We first recall some \ccgreen{definitions} used in \ccred{Section 2.5 and Section 3.4} of \cite{BBS13}. Let $X$ be a quasi-projective threefold. Then $\operatorname{Hilb}^n(X)$ admits a stratification 
$$\operatorname{Hilb}^n(X)=\coprod_{\gamma\vdash n}\operatorname{Hilb}^n_\gamma(X),$$
where $\gamma$ is a partition of $n$, and $\operatorname{Hilb}^n_\gamma(X)$ is the locally closed subscheme of $\operatorname{Hilb}^n(X)$ parametrizing length $n$ subschemes whose support multiplicities are given by $\gamma$. In particular, there are $\gamma_i$ length $i$ clusters in the length $n$ subscheme. 

First consider the case of $\operatorname{Hilb}^n(\mathbb{C}^3)$. Define the relative motivic class of $\operatorname{Hilb}^n(\mathbb{C}^3)$ in $\mathcal{M}_{\operatorname{Hilb}^n(\mathbb{C}^3)}^{\hat{\mu}}$ by:
\begin{equation}
[\operatorname{Hilb}^n(\mathbb{C}^3)]_{relvir}=\mathbb{L}^{-dim(\operatorname{NHilb}^n(\mathbb{C}^3)/2)}\odot MF_{\operatorname{NHilb}^n(\mathbb{C}^3), W}^\phi.
\end{equation} 
Comparing to Theorem \ref{thm1.1}, we see that this is the motivic invariant associated to $(\operatorname{Hilb}^n(\mathbb{C}^3), W)$ without the contribution of a principal $\mathbb{Z}_2$\ccgreen{-bundle} from the orientation. 
Let $[\operatorname{Hilb}^n_\gamma(\mathbb{C}^3)]_{relvir}$ to be the pullback of $[\operatorname{Hilb}^n(\mathbb{C}^3)]_{relvir}$ to $\operatorname{Hilb}^n_\gamma(\mathbb{C}^3)$. On the deepest strata, there is an embedding 
\begin{equation}\label{eq:translation}
\{0\}\times \operatorname{Hilb}^n(\mathbb{C}^3)_0\subset \mathbb{C}^3\times \operatorname{Hilb}^n(\mathbb{C}^3)_0\simeq \operatorname{Hilb}^n_{(n)}(\mathbb{C}^3)
\end{equation}
where $\operatorname{Hilb}^n_{(n)}(\mathbb{C}^3)$ is the punctual Hilbert scheme.   The isomorphism 
in (\ref{eq:translation}) is given by sending $(p,Z)\in \mathbb{C}^3\times \operatorname{Hilb}^n(\mathbb{C}^3)_0$ to the subscheme $Z+p\subset\mathbb{C}^3$ obtained by translating $Z$ by $p$.

Let $[\operatorname{Hilb}^n(\mathbb{C}^3)_0]_{relvir}$ be the pullback of $[\operatorname{Hilb}^n_{(n)}(\mathbb{C}^3)]_{relvir}$ to $\operatorname{Hilb}^n(\mathbb{C}^3)_0$. Also define the absolute motivic classes to be the pushforwards of the corresponding relative motivic classes to a point, and denote them by $[\operatorname{Hilb}^n(\mathbb{C}^3)]$, $[\operatorname{Hilb}^n_\gamma(\mathbb{C}^3)]$ and $[\operatorname{Hilb}^n(\mathbb{C}^3)_0]$ respectively.

Recall the following proposition and definition in \cite{BBS13}:
\begin{proposition} (\cite{BBS13} Proposition 3.6)
\label{prop13}
(1) The absolute motivic classes $[\operatorname{Hilb}^n_\gamma(\mathbb{C}^3)]$ and $[\operatorname{Hilb}^n_\gamma(\mathbb{C}^3)_0]$ live in the subring $\mathcal{M}_\mathbb{C}\subset \mathcal{M}_\mathbb{C}^{\hat{\mu}}$.

(2) On the closed stratum, 
\begin{equation*}
[\operatorname{Hilb}^n_{(n)}(\mathbb{C}^3)]=\mathbb{L}^3\cdot[\operatorname{Hilb}^n(\mathbb{C}^3)_0]\in \mathcal{M}_\mathbb{C}.
\end{equation*}

(3) More generally, for a general stratum, 
\begin{equation*}
[\operatorname{Hilb}^n_\gamma(\mathbb{C}^3)]=\pi_{G_\gamma}([\prod_i(\mathbb{C}^3)^{\gamma_i}\setminus\Delta]\cdot\prod_i[\operatorname{Hilb}^i(\mathbb{C}^3)_0^{\gamma_i}]).
\end{equation*}
where the map $\pi_{G_\gamma}$ is defined by taking the orbit space on generators. 
\end{proposition}

\begin{definition}(\cite{BBS13} Definition 4.1)
\label{def4}
We define motivic classes
$[\operatorname{Hilb}^n_\gamma(X)]\in\mathcal{M}_\mathbb{C}$ and $[\operatorname{Hilb}^n(X)]\in\mathcal{M}_\mathbb{C} $
as follows.

(1) on the deepest stratum, 
\begin{equation*}
[\operatorname{Hilb}^n_{(n)}(X)]=[X]\cdot[\operatorname{Hilb}^n(\mathbb{C}^3)_0].
\end{equation*}

(2) More generally, on all strata, 
\begin{equation*}
[\operatorname{Hilb}^n_\gamma(X)]=\pi_{G_\gamma}([\prod_i X^{\gamma_i}\setminus\Delta]\cdot\prod_i[\operatorname{Hilb}^i(\mathbb{C}^3)_0^{\gamma_i}]).
\end{equation*}

(3) Finally
\begin{equation*}
[\operatorname{Hilb}^n(X)]=\sum_\gamma[\operatorname{Hilb}^n_\gamma(X)].
\end{equation*}

\end{definition}

Now consider the case when $X=\omega_S$ for $S=\mathbb{P}^2$ or $\mathbb{F}_n$. Let $[\operatorname{Hilb}^n(\omega_S)]$ be defined as in Definition \ref{def4}.  Then we have
\begin{theorem}
\label{thm3.2}
Given the d-critical locus $(\operatorname{Hilb}^n(\omega_S), s)$ and the choice of orientation 
$$K_{\operatorname{Hilb}(\omega_S),s}^{1/2}= \operatorname{det}(\pi_{1*}F)^*,$$ 
the absolute motivic class of the motive given by Theorem \ref{thm1.1} concides with
$[\operatorname{Hilb}^n(\omega_S)]$.
\end{theorem}

\begin{proof}
Since we are given a d-critical locus structure with a choice of orientation, by Theorem \ref{thm1.1} there is a unique globally well defined motivic DT invariant $MF_{\operatorname{Hilb}^n{\omega_S}, s}$.  We continue our practice of using a generic index $\gamma$ to be common shorthand for the indexing of either $\mathbb{P}^2$ or $\mathbb{F}_n$.  Furthermore, for any critical chart $(R_\alpha, U_\alpha, f_\alpha, i_\alpha)=(\operatorname{Hilb}^n(V_\alpha), N_\alpha, W_\alpha, i_\alpha)$ this invariant satisfies:
\begin{equation*}
MF_{\operatorname{Hilb}^n(\omega_S),s}|_{\operatorname{Hilb}^n(V_\alpha)}=i_\alpha^*(\mathbb{L}^{-dim{N_\alpha}/2}\odot MF^\phi_{N_\alpha,W_\alpha})\odot\Upsilon(Q_{\operatorname{Hilb}^n(V_\alpha), N_\alpha, W_\alpha, i_\alpha})\in\overline{\mathcal{M}}_{\operatorname{Hilb}^n(V_\alpha)}^{\hat{\mu}}.
\end{equation*} 

\ccgreen{By the construction of $K_{\operatorname{Hilb}^n(\omega_S), s}$, we have 
\begin{equation*}
	\iota_{\operatorname{Hilb}^n(V_\alpha), N_\alpha, i_\alpha, W_\alpha}=\tau_{\operatorname{Hilb}^n(V_\alpha), N_\alpha, i_\alpha, W_\alpha}^2: K_{(\omega_S), s}|_{\operatorname{Hilb}^n(V_\alpha)^{red}}\rightarrow K_{N_\alpha}^2|_{\operatorname{Hilb}^n(V_\alpha)^{red}},
\end{equation*}
where $\tau_{\operatorname{Hilb}^n(V_\alpha), N_\alpha, i_\alpha, W_\alpha}$ is the canonical isomorphism
(\ref{eq:localor}).
Hence the map 
\begin{equation*}\alpha_{\operatorname{Hilb}^n(V_\alpha)}: K^{1/2}_{\operatorname{Hilb}^n(\omega), s}|_{\operatorname{Hilb}^n(V_\alpha)^{red}}\rightarrow i^*_{\alpha}(K_{N_\alpha})|_{\operatorname{Hilb}^n(V_\alpha)^{red}}
\end{equation*} 
in Theorem $\ref{thm1.1}$ is identified with $\tau_{\operatorname{Hilb}^n(V_\alpha), N_\alpha, i_\alpha, W_\alpha}$.
It follows that} $Q_{\operatorname{Hilb}^n(V_\alpha), N_\alpha, W_\alpha, i_\alpha}$ and therefore $\Upsilon(Q_{\operatorname{Hilb}^n(V_\alpha), N_\alpha, W_\alpha, i_\alpha})$ is trivial for each $\alpha$, implying
\begin{equation*}
\label{equMF}
MF_{\operatorname{Hilb}^n(\omega_S),s}|_{\operatorname{Hilb}^n(V_\alpha)}=[\operatorname{Hilb}^n(V_\alpha)]_{relvir}.
\end{equation*}
By Proposition \ref{prop13}, $[\operatorname{Hilb}^n(V_\alpha)]_{relvir}\in \mathcal{M}_{\operatorname{Hilb}^n(V_\alpha)}$. 
Since the $\operatorname{Hilb}^n(V_\alpha)$ cover $\operatorname{Hilb}^n(\omega_S)$, 
it follows that $MF_{\operatorname{Hilb}^n(\omega_S),s}\in \mathcal{M}_{\operatorname{Hilb}^n(\omega_S)}$.   

Hence we have 
\begin{equation}
\label{decom}
MF_{\operatorname{Hilb}^n(\omega_S),s}=\sum_\gamma MF_{\operatorname{Hilb}^n(\omega_S),s}|_{\operatorname{Hilb}^n_\gamma(\omega_S)}.
\end{equation}

Since 
\begin{equation*}
\operatorname{Hilb}^n_{(n)}(V_\alpha)\simeq V_\alpha\times \operatorname{Hilb}^n(\mathbb{C}^3)_0
\end{equation*}
and $W$ is translation invariant in the sense that
\begin{equation*}
\operatorname{tr}([X,Y]Z)=\operatorname{tr}([X+xId,Y+yId](Z+zId))
\end{equation*}
for any $(x,y,z)\in \mathbb{C}^3$, we see that
\begin{equation*}
[\operatorname{Hilb}^n_{(n)}(V_\alpha)]_{relvir}=p_2^*[\operatorname{Hilb}^n(\mathbb{C}^3)_0]_{relvir},
\end{equation*}
where $p_2:V_\alpha\times \operatorname{Hilb}^n(\mathbb{C}^3)_0\to \operatorname{Hilb}^n(\mathbb{C}^3)_0$ is the projection to the second factor. 
Then on the deepest stratum, we have
\begin{equation*}
\label{equn}
MF_{\operatorname{Hilb}^n(\omega_S),s}|_{\operatorname{Hilb}^n_{(n)}(V_\alpha)}=[\operatorname{Hilb}^n_{(n)}(V_\alpha)]_{relvir}=[V_\alpha\times \operatorname{Hilb}^n(\mathbb{C}^3)_0]\cdot_{\operatorname{Hilb}^n(\mathbb{C}^3)_0}[\operatorname{Hilb}^n(\mathbb{C}^3)_0]_{relvir}.
\end{equation*}
We use the subscript to indicate the scheme over which the product of relative motivic classes takes place. Since the motivic invariant $MF_{\operatorname{Hilb}^n(\omega_S),s}$ is uniquely determined by its restriction to the $\operatorname{Hilb}^n(V_\alpha)$, we have 
\begin{equation*}
MF_{\operatorname{Hilb}^n(\omega_S),s}|_{\operatorname{Hilb}^n_{(n)}(\omega_S)}=[\omega_S\times \operatorname{Hilb}^n(\mathbb{C}^3)_0]\cdot_{\operatorname{Hilb}^n(\mathbb{C}^3)_0}[\operatorname{Hilb}^n(\mathbb{C}^3)_0]_{relvir}.
\end{equation*}
Now consider a general stratum. Let $Y_\gamma\subset\prod_i\operatorname{Hilb}^i_{(i)}(\mathbb{C}^3)^{\gamma_i}$ be the open subset on which the clusters have distinct supports.
We have
\begin{equation*}
\label{equa}
\begin{split}
MF_{\operatorname{Hilb}^n(\omega_S),s}|_{\operatorname{Hilb}^n_\gamma(V_\alpha)}=&[\operatorname{Hilb}^n_\gamma(V_\alpha)]_{relvir}\\ 
=&\pi_{G_\gamma}(\prod_i[\operatorname{Hilb}^i_{(i)}(\mathbb{C}^3)^{\gamma_i}]_{relvir}|_{Y_\gamma}).
\end{split}
\end{equation*}
We also have a fiber product expression
\begin{equation*}
\prod_i[\operatorname{Hilb}^i_{(i)}(\mathbb{C}^3)^{\gamma_i}]_{relvir}=\prod_i[\left(\mathbb{C}^3\times \operatorname{Hilb}^i(\mathbb{C}^3)_0\right)^{\gamma_i}]\cdot_{\operatorname{Hilb}^n_{\gamma,0}(\mathbb{C}^3)}[(\operatorname{Hilb}^n_{\gamma,0}(\mathbb{C}^3)]_{relvir}.
\end{equation*}
where we have put $\operatorname{Hilb}^n_{\gamma,0}(\mathbb{C}^3)=\prod_i \left(\operatorname{Hilb}^i(\mathbb{C}^3)_0\right)^{\gamma_i}$.

Then we get
\begin{equation}\label{eq:mfstrata}
MF_{\operatorname{Hilb}^n(\omega_S),s}|_{\operatorname{Hilb}^n_\gamma(\omega_S)}=\pi_{G_\gamma}\left(\left[\left(\prod_i \omega_S^{\gamma_i}-\Delta\right)\times ^n_{\gamma,0}(\mathbb{C}^3)\right]\cdot[(\operatorname{Hilb}^n_{\gamma,0}(\mathbb{C}^3)]_{relvir}\right),
\end{equation}
where $\Delta\subset\prod_i \omega_S^{\gamma_i}$ is the big diagonal.

Taking the absolute motivic class of (\ref{decom}) and using (\ref{eq:mfstrata}), we see that the absolute motive of $MF_{\operatorname{Hilb}^n(\omega_S),s}$ matches the motivic class of \cite{BBS13} from Definition \ref{def4}, completing the proof.
\end{proof}

These absolute motives can be combined into a generating function
\begin{equation*}
Z_{\omega_S}(t)=\sum_{n=0}^\infty MF_{\operatorname{Hilb}^n(\omega_S),s}t^n.
\end{equation*}
Let $\mathrm{Exp}$ denote the plethystic exponential.

\begin{corollary}
\label{cor:toricBBS}
$$Z_{\omega_S}(t)=
\left\{
\begin{array}{cc}
\mathrm{Exp}\left(
\frac{\mathbb{L}^{-1/2}(\mathbb{L}^2+\mathbb{L}+1)t}{\left(1-\mathbb{L}^{1/2}t\right)\left(1-\mathbb{L}^{-1/2}t\right)}
\right)& S=\mathbb{P}^2\\
\mathrm{Exp}\left(
\frac{\mathbb{L}^{-1/2}(\mathbb{L}^2+2\mathbb{L}+1)t}{\left(1-\mathbb{L}^{1/2}t\right)\left(1-\mathbb{L}^{-1/2}t\right)}
\right)& S=\mathbb{F}_n.
\end{array}
\right. 
$$

\begin{proof}
Follows immediately from \cite[Theorem 4.3]{BBS13}, Theorem~\ref{thm3.2},  $[\omega_S]=\mathbb{L}[S]$ for any $S$, $[\mathbb{P}^2]=\mathbb{L}^2+\mathbb{L}+1$, and $[\mathbb{F}_n]=\mathbb{L}^2+2\mathbb{L}+1$.
\end{proof}
\end{corollary}

\bigskip\noindent
{\bf Remark.} We expect that our methods are adaptable to other local toric Calabi-Yaus.  We leave this investigation for future work.

\section*{Acknowledgements}
We thank Ben Davison for helpful discussions on this subject and motivic DT theory in general and Tony Pantev for helpful conversations about $-1$-shifted symplectic structures. We are also grateful for Bal\'azs Szendr\H{o}i for helpful comments on a previous version of this paper. The research of the first author is supported in part by NSF grants DMS--1802242 and DMS--2201203, as well as by NSF grant DMS-1440140 while in residence at MSRI in Spring, 2018. Both authors would like to thank MSRI for the excellent working environment.  The second author is grateful for the mini course given by Kai Behrend during the program Enumerative Geometry beyond Numbers at MSRI, where she first learned this problem. Part of the work was done while the second author was a postdoc at CMSA, Harvard.  She would like to thank CMSA for the excellent working environment.   

\bibliographystyle{mrl}
\bibliography{References}

\begin{thebibliography}{10}

\bibitem{Beh09}
K.~Behrend, \emph{Donaldson-{T}homas type invariants via microlocal geometry},
  Ann. of Math. (2) \textbf{170} (2009), no.~3,  1307--1338.

\bibitem{BBS13}
K.~Behrend, J.~Bryan, and B.~Szendr\H{o}i, \emph{Motivic degree zero
  {D}onaldson-{T}homas invariants}, Invent. Math. \textbf{192} (2013), no.~1,
  111--160.

\bibitem{BF97}
K.~Behrend and B.~Fantechi, \emph{The intrinsic normal cone}, Invent. Math.
  \textbf{128} (1997), no.~1,  45--88.

\bibitem{BF08}
---{}---{}---, \emph{Symmetric obstruction theories and {H}ilbert schemes of
  points on threefolds}, Algebra Number Theory \textbf{2} (2008), no.~3,
  313--345.

\bibitem{BD19}
C.~Brav and T.~Dyckerhoff, \emph{Relative {C}alabi-{Y}au structures {II}:
  shifted {L}agrangians in the moduli of objects}, Selecta Math. (N.S.)
  \textbf{27} (2021), no.~4,  Paper No. 63, 45.

\bibitem{BJM19}
V.~Bussi, D.~Joyce, and S.~Meinhardt, \emph{On motivic vanishing cycles of
  critical loci}, J. Algebraic Geom. \textbf{28} (2019), no.~3,  405--438.

\bibitem{CLS11}
D.~A. Cox, J.~B. Little, and H.~K. Schenck, Toric varieties, Vol. 124 of
  \emph{Graduate Studies in Mathematics}, American Mathematical Society,
  Providence, RI (2011), ISBN 978-0-8218-4819-7.

\bibitem{Dav}
B.~Davison, \emph{Invariance of orientation data for ind-constructible
  Calabi-Yau $A_\infty$ categories under derived equivalence}. Preprint,
  arXiv:1006.5475.

\bibitem{DF99}
J.~Denef and F.~Loeser, \emph{Motivic exponential integrals and a motivic
  {T}hom-{S}ebastiani theorem}, Duke Math. J. \textbf{99} (1999), no.~2,
  285--309.

\bibitem{DF00}
---{}---{}---, \emph{Geometry on arc spaces of algebraic varieties}, in
  European {C}ongress of {M}athematics, {V}ol. {I} ({B}arcelona, 2000), Vol.
  201 of \emph{Progr. Math.}, 327--348, Birkh\"{a}user, Basel (2001), ISBN
  3-7643-6417-3.

\bibitem{Joy15}
D.~Joyce, \emph{A classical model for derived critical loci}, J. Differential
  Geom. \textbf{101} (2015), no.~2,  289--367.

\bibitem{KS22}
S.~Katz and Y.~Shi, \emph{{$D$}-critical loci for length {$n$} sheaves on local
  toric {C}alabi-{Y}au 3-folds}, Bull. Lond. Math. Soc. \textbf{54} (2022),
  no.~6,  2101--2116.

\bibitem{KS}
M.~Kontsevich and Y.~Soibelman, \emph{Stability structures, motivic
  Donaldson-Thomas invariants and cluster transformations}. Preprint,
  arXiv:0811.2435.

\bibitem{LP09}
M.~Levine and R.~Pandharipande, \emph{Algebraic cobordism revisited}, Invent.
  Math. \textbf{176} (2009), no.~1,  63--130.

\bibitem{Loo02}
E.~Looijenga, \emph{Motivic measures}, 276, 267--297 (2002). S\'{e}minaire
  Bourbaki, Vol. 1999/2000.

\bibitem{MNOP06}
D.~Maulik, N.~Nekrasov, A.~Okounkov, and R.~Pandharipande,
  \emph{Gromov-{W}itten theory and {D}onaldson-{T}homas theory. {I}}, Compos.
  Math. \textbf{142} (2006), no.~5,  1263--1285.

\bibitem{MMNS12}
A.~Morrison, S.~Mozgovoy, K.~Nagao, and B.~Szendr\H{o}i, \emph{Motivic
  {D}onaldson-{T}homas invariants of the conifold and the refined topological
  vertex}, Adv. Math. \textbf{230} (2012), no. 4-6,  2065--2093.

\bibitem{Nag}
K.~Nagao, \emph{Wall-crossing of the motivic {D}onaldson-{T}homas invariants}.
  Preprint, arXiv:1103.2922.

\bibitem{PTVV13}
T.~Pantev, B.~To\"{e}n, M.~Vaqui\'{e}, and G.~Vezzosi, \emph{Shifted symplectic
  structures}, Publ. Math. Inst. Hautes \'{E}tudes Sci. \textbf{117} (2013)
  271--328.

\bibitem{PP98}
A.~Parusi\'{n}ski and P.~Pragacz, \emph{Characteristic classes of hypersurfaces
  and characteristic cycles}, J. Algebraic Geom. \textbf{10} (2001), no.~1,
  63--79.

\bibitem{Shi18}
Y.~Shi, \emph{Orientation data for local $\mathbb{P}^2$}. ArXiv:1809.01776, to
  appear in Math. Res. Lett.

\bibitem{Tho00}
R.~P. Thomas, \emph{A holomorphic {C}asson invariant for {C}alabi-{Y}au
  3-folds, and bundles on {$K3$} fibrations}, J. Differential Geom. \textbf{54}
  (2000), no.~2,  367--438.

\end{thebibliography}

\end{document}